\documentclass[11pt,reqno]{amsart}
\usepackage{amssymb,mathrsfs,graphicx,enumerate}
\usepackage{amsmath,amsfonts,amssymb,amscd,amsthm,bbm}
\usepackage{subcaption}
\usepackage{colortbl}
 \usepackage{kotex}
\usepackage{mathtools}

\graphicspath{{./figure/}}

\makeatletter
\@namedef{subjclassname@2020}{\textup{2020} Mathematics Subject Classification}
\makeatother

\title[CS model with velocity control]{Asymptotic dynamics for the Cucker-Smale model with velocity control}

\author[Byeon]{Junhyeok Byeon}
\address[Junhyeok Byeon]{\newline Department of Mathematical Sciences\newline Seoul National University, Seoul 08826, Republic of Korea}
\email{giugi2486@snu.ac.kr}

\newtheorem{theorem}{Theorem}[section]
\newtheorem{lemma}{Lemma}[section]
\newtheorem{corollary}{Corollary}[section]
\newtheorem{proposition}{Proposition}[section]
\newtheorem{remark}{Remark}[section]
\newtheorem{example}{Example}[section]

\newtheorem{definition}{Definition}[section]

\newcommand{\bbr}{\mathbb R}

\usepackage{verbatim}

\begin{document}

\date{\today}


%\thanks{\textbf{Acknowledgment.} The work of S.-Y. Ha is supported by NRF-2020R1A2C3A01003881.}

\begin{abstract}
We study the Cucker-Smale model with a velocity control function. The Cucker-Smale model design the emergence of consensus in terms of flocking. A proposed model encompasses several Cucker-Smale models, such as a speed limit model, a relativistic model, and an almost unit speed model. We provide collective behaviors of the proposed model, like mono or bi-cluster flocking, sticking, and collision avoidance, depending on the regularity and singularity of communication weight at the origin. In particular, we provide a sufficient framework to guarantee a positive lower bound of the distance between agents under strongly singular communications. 
\end{abstract}
\maketitle \centerline{\date}

%\tableofcontents

	\section{Introduction} \label{sec:1}
\setcounter{equation}{0} 

Collective behaviors of complex systems are ubiquitous. In \cite{HPZ18}, to name a few, herding of sheep, schooling of fish, and synchronization of fireflies \cite{AB19,C-F-T-V,M-T1}, etc. Thanks to potential engineering applications for unmanned aerial vehicles (UAVs), robotics, and client network equipment, collective dynamics of many-body systems have been extensively investigated in engineering domains, such as control theory. Among others, the Cucker-Smale (in short, CS) model describes the dynamics of self-propelled particles' flocking behaviors, and it received lots of attention as it unites seemingly unrelated phenomena \cite{AB19,P-R-K,T-T,VZ}. For the mathematical analysis, kinetic description, and hydrodynamic description of the CS model, we refer to \cite{AB19,CFRT10,CHL17,HL09,HT08,M-T1}, \cite{CCTT16}, and \cite{K21}, respectively. 

Let $q_i$ be the position of the $i$-th agent (particle). In this paper, we are interested in the CS model with \textit{velocity control}. The following system of ODEs governs the agents of the proposed model:
\begin{align}\label{A-3'}
	\begin{cases}
		\dot{q}_i= G(p_i),\quad t>0, \quad i \in [N]:=\{1,2,\cdots,N\}, \vspace{.2cm}\\
		\displaystyle \dot{p}_i=\frac{\kappa}{N}\sum_{k=1}^N\psi(|q_k-q_i|)(G(p_k)-G(p_i)) \vspace{.2cm},\\
		(q_i,p_i) \big|_{t = 0+} = (q_i^0,p_i^0), \quad p_i,q_i \in \bbr^d,
	\end{cases}
\end{align}
where $\psi$ and $\kappa>0$ represent the communication (kernel) between the agents and its intensity, respectively. Above, $G:\bbr^d \to I \subset \bbr^d$ is a velocity control function, and $G(p_i)$ is interpreted as the velocity of an $i$-th particle. In particular, when $G$ is an identity function, the model \eqref{A-3'} reduces to the standard CS model. We assume that $G=g(|p|)p/|p|$, where $|p| \mapsto g(|p|)$ is a convex or concave function which is continuously differentiable and increasing with a non-vanishing gradient (for the detail, refer to \eqref{G}). In the model \eqref{A-3'}, the presence of $G$ may lead to a huge influence on the agent’s behavior. For example, if $G$ is a bounded function, then its boundedness is realized as a speed limit of agents. Other than the speed limitation, effects like the relativistic effect or almost unity of speed can be implemented. For more details, we refer to Example \ref{E2.1}.

One of the main purpose of this paper is to confirm whether the model \eqref{A-3'} exhibits consistent behavior with the standard CS model under a regular ($\lim_{x \searrow 0}\psi(x)<\infty$) or singular ($\lim_{x \searrow 0}\psi(x)=\infty$) communication despite velocity control. More precisely, we consider the following three kind of communications:
\begin{align*}
 \mbox{Type I}&:~~ \psi \in (L^\infty \cap C^{0,1})(\bbr_+;\bbr_+), \\
 &\hspace{0.5cm}(\psi(r)-\psi(s))(r-s) \leq 0, \quad \forall r,s \in \bbr_+, \\
 \mbox{Type II}&:~~\psi(q)=\frac{1}{|q|^\alpha}, \quad \alpha \in (0,1),\quad q \neq 0, \\
 \mbox{Type III}&:~~\psi(q)=\frac{1}{|q|^\alpha}, \quad \alpha \geq 1,\quad q \neq 0,
\end{align*}
where $\bbr_+$ is a set of positive real numbers, and $(L^\infty \cap C^{0,1})(\bbr_+;\bbr_+)$ is a space of bounded and Lipschitz continuous functions from $\bbr_+$ to $\bbr_+$, respectively. For any type of kernel, under appropriate initial conditions we can expect flocking. However, for $\psi$ of type II and III, the vector field associated with \eqref{A-3'} becomes unbounded at the instant that two particles collide. Therefore the existence and uniqueness of of a solution cannot be treated by the standard Cauchy-Lipschitz theory. In previous work \cite{P15,P14}, the authors proved the existence and uniqueness of weak solution in the Sobolev space $W^{2,1}$ for type II kernel under $\alpha \in (0,1/2)$ and $G=\mathrm{Id}$. In this case, two agents may \emph{stick} in finite time. On the other hand, the well-definedness of a solution for type III kernel is guaranteed by \emph{collision avoidance} induced by strong singularity ($\int_0^\varepsilon \psi(x)dx = \infty$) of communications \cite{CCMP17}. Such collision avoidance property is interesting in view of practical applications, and also desirable from a theoretical point of view (see Remark \ref{R4.2}). The second purpose of this paper is to present more improved results on the well-definedness and collision avoidance properties of solutions under singular communications.

Our main results of this paper are three-fold. First, for the coupling of type I, we investigate some equivalence conditions for the emergence of flocking. Let $(P,Q)$ be a solution to \eqref{A-3}, and define $D_P(t)$ and $D_Q(t)$ as maximum of $|p_i(t)-p_j(t)|$ and $|q_i(t)-q_j(t)|$, respectively. We say that $(P,Q)$ exhibits (mono-cluster) flocking if agents form a single group ($\sup_{t \geq 0} D_Q(t) < \infty$) and they align their velocity ($\lim_{t \to \infty} D_P(t) =0$). Then we have:\begin{align}\begin{aligned}\label{AX-1}
	(P,Q) \text{ exhibits flocking } ~ &\Leftrightarrow ~ D_P(t) \text{ decays exponentially } \\
	~ &\Leftrightarrow ~ D_Q(t) \text{ is bounded. }
\end{aligned}\end{align}
Based on the above equivalence, we provide a sufficient condition to achieve flocking in terms of initial data and system parameters:
\begin{align}\label{AX-2}
		\|P^0\| < \frac{\mathcal{M}\kappa}{M_{G'}}\int_{\|Q^0\|}^{+\infty}\psi(s)\,ds,
\end{align}
where $\|Q^0\|^2:=\sum_{i,j \in [N]}|q_i^0-q_j^0|^2$,  $\|P^0\|^2:=\sum_{i,j \in [N]}|p_i^0-p_j^0|^2$, and $\mathcal{M}$, $M_{G'}$ are positive constants characterized by initial data and velocity control function $G$. As an application, we also provide an equivalent condition for a bi-cluster flocking as well.

Our second result deal with the existence, uniqueness, and regularity of the solution for communication type II on the \emph{real line}. More precisely, when the ambient space of \eqref{A-3'} is $\bbr$ and the kernel is of the form $\psi(x)={|x|}^{-\alpha} ~ (\alpha \in (0,1))$, \eqref{A-3'} admits a unique global weak solution $(P,Q)$, where $q_i$ is in the Sobolev space $W^{2,\gamma}([0,T])$ for each $i \in [N]$, $T \in \bbr_+$, and
		\[
			\gamma \in \bigg[1, \frac{1}{\max\{ 1-K, \alpha\}}\bigg), \quad 
			K := C\left(\frac{1}{\alpha}-1\right),
		\] 
where $C$ is a positive constant depending on $G$ and $\alpha$. This states that a weak solution is well defined for any $\alpha \in (0,1)$ and its regularity improves as $\alpha$ decrease. In particular, $p_i$ is close to Lipchitz continuous function when $\alpha$ is close to $0$. Note that in this case, $(P,Q)$ always exhibits flocking unconditionally, and agents may stick in finite time.

Our third result deal with a collision avoidance property under the kernel of type III: if initial data is non-collisional, then there exists a unique global classical solution with the collision avoidance property:
	\[
		\inf_{t \in [0,T]}\min_{\substack{i,j \in [N] \\ i \neq j}}|q_i(t)-q_j(t)| > 0, \quad \forall T \in \bbr_+.
	\]
Furthermore, if we further assume that $\alpha \neq 1$ and flocking emerges (i.e. any of \eqref{AX-1} or \eqref{AX-2} is achieved), then a strictly positive lower bound of relative distance exists between agents:
	\[
		\inf_{t \geq 0}\min_{\substack{i,j \in [N] \\ i \neq j}}|q_i(t)-q_j(t)| \geq L_\infty > 0.
	\]
Therefore for arbitrary initial data and coupling intensity $\kappa > 0$, one can guarantee the existence of positive $L_\infty$ whenever $\alpha$ is sufficiently close to 1. To the authors' knowledge, the existence of $L_\infty$ under a strongly singular kernel is partially solved, and the above results may complement the previous result (see Remark \ref{R4.2}).

The rest of the paper is organized as follows. In Section \ref{sec:2}, we clarify the conditions of the velocity control function and provide several examples. We then study the emergence of flocking dynamics of \eqref{A-3'} for a regular kernel (type I). In Section \ref{sec:3}, we study the existence, uniqueness, and regularity of \eqref{A-3'} for a weakly singular kernel (type II) on the real line. In Section \ref{sec:4}, we review the collision avoidance property of \eqref{A-3'} for a strongly singular kernel (type III) and provide a sufficient condition to guarantee a strictly positive lower bound of the relative distance between the agents. Finally, Section \ref{sec:5} is devoted to a brief summary of our main results and some remaining problems.

\vspace{0.5cm}

\noindent {\bf Notation.} We use the following  handy notation throughout the paper:
\begin{align*}
&\psi_{ij}:=\psi(|q_i-q_j|), \quad [n]:=\{1,2,\cdots,n\} ~ \text{for} ~ n \in \mathbb{N}, \\
&\bbr_{\geq 0}:= \{ x \mid x \geq 0\}, \quad \bbr_+ := \{ x \mid x>0 \}.
\end{align*}
For configuration vectors $q_i \in \bbr^d$ and $p_i \in \bbr^d$, we denote
 \begin{align*}
 	&Q(t) := (q_1(t), \ldots, q_N(t)), \quad P(t) := (p_1(t), \ldots, p_N(t)), \\
	&Q^0 := Q(0), \quad P^0 := P(0), \quad \mathcal{N}:=(\nu_1,\nu_2,\cdots,\nu_N),
\end{align*}
where $\nu_i$ is to be defined in Section \ref{sec:3}. For $S \subset [N]$, we define the norms on a (sub)system of $\{p_i-p_j\}_{i,j \in [N]}$ and $\{q_i-q_j\}_{i,j \in [N]}$ as
\begin{align*}
	&\|Q\|_{S} := \sqrt{\sum_{i,j \in S}|q_i-q_j|^2}, \quad \|P\|_{S} := \sqrt{\sum_{i,j \in S}|p_i-p_j|^2}, \\
	& \|Q\|:=\|Q\|_{[N]} \quad \|P\|:=\|P\|_{[N]}, \\
	&D_{P,S} := \max_{i,j \in S} | p_i - p_j |, \quad
	D_{Q,S} := \max_{i,j \in S} | p_i - p_j |, \\
	&D_P := D_{P,[N]}, \quad D_Q := D_{Q,[N]}.
\end{align*}
Space of function from $X$ to $Y$ is written as $\mathcal{F}(X;Y)$. If $Y=\bbr$, we write $\mathcal{F}(X):=\mathcal{F}(X;Y)$. Quantity $\|f\|_{L^p(X)}$, the $L^p$ norm of $f$ on $X$, is defined as $(\int_X f(x) dx)^{1/p}$ for $p \in [1,\infty)$, and defined as an essential supremum of $f$ for $p=\infty$. If $L^p$ norm of $f$ on $X$ is finite, it is denoted by $f \in L^p(X)$. $C^{0,1}(X)$ is a space of Lipschitz continuous function on $X$. Sobolev space $W^{k,p}(X)$ for $X \subset \bbr$ is defined as a space of functions $f \in L^p(X)$ such that $f$ and its weak derivatives up to order $k$ have a finite $L^p$ norm on $X$. If $f \in \mathcal{F}(C)$ for any compact subset $C$ of $X$, we write $f \in \mathcal{F}_{\mathrm{loc}}(X)$.
%%%%%%
\begin{comment}
\begin{remark} Role of activation function
\begin{enumerate}
	\item good in terms of training
	\item physical meaning: relativity
	\item finiteness in compact interval?
\end{enumerate}
\end{remark}
\end{comment}
%%%%%%%

\section{Analysis under regular communications}\label{sec:2}
\setcounter{equation}{0}
In this section, we study the basic properties of \eqref{A-3'} and present technical lemmas. Based on the aforementioned results, we provide several conditions for mono or bi-cluster flocking. 

\subsection{Velocity control function.} In this subsection, we impose several conditions on a velocity control function $G$ and study its implication to the dynamics. We recall the CS model with velocity control:
\begin{align}\label{A-3}
	\begin{cases}
		\dot{q}_i= G(p_i),\quad t>0, \quad i \in [N], \vspace{.2cm}\\
		\displaystyle \dot{p}_i=\frac{\kappa}{N}\sum_{k=1}^N\psi(|q_k-q_i|)(G(p_k)-G(p_i)) \vspace{.2cm},\\
		(q_i,p_i) \big|_{t = 0+} = (q_i^0,p_i^0), \quad p_i,q_i \in \bbr^d,
	\end{cases}
\end{align}
where $\kappa>0$. Throughout Section \ref{sec:2}, we assume
\[
	\psi \in (L^\infty \cap C^{0,1})(\bbr_+;\bbr_+), \quad  (\psi(r)-\psi(s))(r-s) \leq 0, \quad \forall r,s \in \bbr_+.
\]
In \eqref{A-3}, $G$ is assumed to be radially symmetric. More precisely, we assume:
\begin{align}\begin{aligned}\label{G}
	&G(p)=
	\begin{cases}
	\displaystyle g(|p|)\frac{p}{|p|}&, \quad \mathrm{if} \quad p \neq 0, \\
	0&, \quad \mathrm{if} \quad p = 0,
	\end{cases}
	\quad g \in C^1(\bbr_{\geq 0}),
	\quad g(0)=0, \\
	&0 < m_{g'} \leq g' \leq M_{g'} ~ \text{on any compact interval}, \\
	&g ~ \text{is convex or concave on $\bbr_+$},
\end{aligned}\end{align}
where $m_{g'}$ and $M_{g'}$ may depends on a compact interval. 

\begin{example}\label{E2.1} We address some possible examples of velocity control function $G$.
	\begin{enumerate}
		\item \textbf{The CS model.} The simplest and most motivating example for $G$ is the identity mapping $G(p)=p$. In this case, system \eqref{A-3} represents a standard CS model \cite{CS07-2,CS07-1}.
		\item \textbf{Speed limit model.} Suppose that $G$ is bounded, say $g(\bbr)=[0,M), ~ M<\infty$. Then we have
		\[
			|\dot{q}_i| = |G(p_i)| < M,
		\]
		and the maximum speed of agents is always bounded by $M$. This may not feature out for the model \eqref{A-3} since maximal speed always decrease in this case (see Proposition \ref{P2.1}). However, despite the presence of extra force, which may increase the maximal speed(e.g., random noise \cite{AH} or bonding force \cite{ABHY}), we can still guarantee the speed limitation.
		\item \textbf{Physical models.} Several physical effects can be reflected by the suitable choice of $G$. For example, if we involve the Lorentz factor $\Gamma$ as follows:
		\[
			g^{-1}:[0,c) \mapsto \bbr, \quad g(v):= \Gamma\left(1+\frac{\Gamma}{c^2}\right)v, \quad \Gamma:=\frac{1}{\sqrt{1-\frac{v^2}{c^2}}},
		\]
		then the model \eqref{A-3} becomes tha relativistic Cucker-Smale (RCS) model, which is introduced as the relativistic correction of the CS model. For the derivation and emergent dynamics of the RCS model, we refer to \cite{AHK,HKR20}. Other than relativistic effects, physical semantics like proper velocity or rapidity can be reflected \cite{BHK2,M-A-H-K}.
		\item \textbf{Almost unit speed model.} In literature, several Vicsek-type models with a unit speed constraint have been studied in terms of the heading angle. For the CS model with unit speed, refer to \cite{CH}. In terms of \eqref{A-3}, this might be represented by choice  of $g_0 \equiv 1$ on $\bbr_{+}$. This does not fulfill \eqref{G}, but can be approximated by functions satisfying \eqref{G}. For example, for $g_\varepsilon(p):=\tanh(p/\varepsilon)$, we expect
		\[
			g_\varepsilon \xrightarrow{\varepsilon \searrow 0} g_0 = 1 ~ \text{on ~ $\bbr_+$},
		\]
		and we formally have a close-to-unit speed model under $\varepsilon \ll 1$. Compared to the model in \cite{CH}, the above model does not strictly have unit speed but has the advantage of applying methodology consistent with the standard CS model. 
		
%		Since a maximum of $p_i(t)$ is always bounded by $P^0_M$ and $\sum_{k=1}^N p_k(t)$ is conserved, this implies
%	\[
%		\left|\sum_{k=1}^N p_k^0\right| = \left|\sum_{k=1}^N p_k(t)\right|
%		\leq |p_i(t)| + (N-1)P_M^0.
%	\]
%	Therefore if initial data is well prepare in a sense that
%	\[
%		(N-1)P_M^0 < \left|\sum_{k=1}^N p_k^0\right|,
%	\]
%	then there exists $P_m^0>0$ such that
%	\[
%		0 < P_m^0 \leq |p_i(t)| \leq P_M^0
%	\]
%	for any $i \in [N]$ and $t \in \bbr_+$. In this case, for arbitrary $\varepsilon>0$, we have
%	\[
%		1-\varepsilon < G(p_i(t))=\dot{q}_i(t) < 1,
%	\]
%	whenever $G(p)=\tanh{p/\delta}$ and $\delta$ is sufficiently small.
%			unit speed model ref 걸자
	\end{enumerate}
\end{example}

Technical reason for conditions \eqref{G} will naturally rise up in the following proposition.
% It turns out that each $|p_i|$ is uniformly bounded in time and we can assume that $M_{G'}$ and $m_{G'}$ depend only on initial data.

\begin{proposition}\label{P2.1}
	Let $(P,Q)$ be a global solution to \eqref{A-3} with initial data $(P^0,Q^0)$. Then the following holds.
	\begin{enumerate}
	\item $\sum_{k=1}^N p_j$ is conserved:
	\[
		\frac{d}{dt}\sum_{k=1}^N p_j(t) = 0.
	\]
	\item Maximum modulus of $p_i$ decrease in time:
	\[
		\max_{i \in [N]}|p_i(t)|  \leq \max_{i \in [N]}|p_i(s)| , \quad 0 \leq s \leq t.		
	\]
	In particular,
	\begin{align}\label{A-5}
		\sup_{t \geq 0}\max_{i \in [N]}|p_i(t)| = \max_{i \in [N]}|p_i^0| =: P^0_M.
	\end{align}
	\item For any $t \in \bbr_+$ and $i,j \in [N]$,
	\begin{align*}
	m_{G'}|p_i(t)-p_j(t)| \leq |G(p_i(t))-G(p_j(t))| \leq M_{G'}|p_i(t)-p_j(t)|,
	\end{align*}
	where
	\[
		M_{G'} := \max \{ g'(|p|) : |p| \leq P^0_M \},
		\quad m_{G'} := \min \{ g'(|p|) : |p| \leq P^0_M \}.
	\]
	\item There exists a positive constant $\mathcal{M}=\mathcal{M}(P^0)>0$ satisfying
	\[
		\mathcal{M}|p_i(t)-p_j(t)|^2
		\leq (p_i(t)-p_j(t)) \cdot (G(p_i(t))-G(p_j(t)))
	\]
	for any $t \in \bbr_+$ and $i,j \in [N]$. %In particular, if an ambient space of \eqref{A-3} is one-dimensional $(d=1)$, then $\mathcal{M}=m_{G'}$. 
	\end{enumerate}
\end{proposition}
\begin{proof} (1) We sum $\eqref{A-3}_2$ over $i \in [N]$ and utilize the index symmetry to see
\[
	\frac{d}{dt}\sum_{k=1}^N p_j
	= \frac{\kappa}{N}\sum_{i,k=1}^N\psi(|q_k-q_i|)(G(p_k)-G(p_i))
	= \frac{\kappa}{N}\sum_{i,k=1}^N\psi(|q_i-q_k|)(G(p_i)-G(p_k))=0.
\] 
	(2) Define $M(t) \in \textrm{argmax}_{ i \in [N]}|p_i(t)|$. Let time $t$ and index $M(t)$ be fixed, and set $\ell=M(t)$. 
Then we have
\[
	\frac{d}{dt}|p_\ell|^2
	= \frac{\kappa}{N} \sum_{k=1}^N \psi(|q_k-q_\ell |)  p_\ell \cdot  (G(p_k)-G(p_\ell))
	\leq 0,
\]
where the inequality holds from the maximality of $M$. This proves \eqref{A-5}. \newline
(3) The Jacobian of $G$ at $p$ is
\[
	G'(p) =
	\begin{cases}
	\displaystyle \frac{g(|p|)}{|p|}\mathrm{Id} + \left( g'(|p|)-\frac{g(|p|)}{|p|} \right)\frac{ p \otimes p}{|p|^2}&, \quad \mathrm{if} \quad p\neq0, \vspace{.2cm} \\
	g'(0)\mathrm{Id}&, \quad \mathrm{if} \quad p=0,
	\end{cases}
\]
Since the eigenvalues of $p \otimes p$ are $0$ and $|p|^2$ up to multiplicity, eigenvalues of $G'$ are
\begin{align*}%\label{A-5-0-1}
	\lambda_1 =
	\begin{cases}
		\displaystyle \frac{g(|p|)}{|p|}, \quad &\text{if} \quad p \neq 0, \vspace{.2cm} \\
		g'(0), \quad &\text{if} \quad p = 0,
	\end{cases} \quad \lambda_2 = g'(|p|)
\end{align*}
Due to symmetry of $G'$, the largest eigenvalue is the operator norm of $G'$, and this is bounded by $M_{G'}$ from (2). Therefore the mean value theorem implies
\[
	|G(p_i)-G(p_j)| \leq M_{G'}|p_i-p_j|.
\]
On the other hand, operator norm of inverse Jacobian $(G^{-1})'$ is $\frac{1}{\min\{\lambda_1,\lambda_2\}}$, which is less or equal to $\frac{1}{m_{G'}}$. Therefore we have
\[
	m_{G'}|p_i-p_j| \leq |G(p_i)-G(p_j)| \leq M_{G'}|p_i-p_j|.
\]
(4) Throughout the proof, without the loss of generality, we assume $|p_i| \geq |p_j|$. First suppose that $g$ is convex on $\bbr_+$, so that $| \cdot | \mapsto \frac{g(|\cdot|)}{|\cdot|}$ is an increasing function. Then,
\begin{align*}
	( p_i-p_j ) \cdot & (G(p_i)-G(p_j) ) \\
&= \frac{g(|p_j|)}{|p_j|}|p_i-p_j|^2 + \left(\frac{g(|p_i|)}{|p_i|} - \frac{g(|p_j|)}{|p_j|}\right)( p_i-p_j ) \cdot  p_i \\
&\geq m_{G'}|p_i-p_j|^2.
\end{align*}
Now suppose that $g$ is concave on $\bbr_+$. Since $g$ is increasing, $g^{-1}$ is convex and this yields
\begin{align*}
	( p_i&-p_j) \cdot ( G(p_i)-G(p_j) ) \\
&= \frac{g^{-1}(|G(p_j)|)}{|G(p_j)|}|G(p_i)-G(p_j)|^2 \\
&\hspace{.5cm} + \left(\frac{g^{-1}(|G(p_i)|)}{|G(p_i)|} - \frac{g^{-1}(|G(p_j)|)}{|G(p_j)|}\right)( G(p_i)-G(p_j)) \cdot G(p_i) \\
&\geq \frac{m_{G'}^2}{M_{G'}}|p_i-p_j|^2.
\end{align*}
We pose $\mathcal{M}:=\min\{{m_{G'}},\frac{m_{G'}^2}{M_{G'}}\}$ to complete the proof. Since $m_{G'}$ and $M_{G'}$ depends only on $P_0$, so is $\mathcal{M}$.
\end{proof}

\begin{remark} ~ 
	\begin{enumerate}
	\item In \eqref{G}, $m_{g'}$ and $M_{g'}$ depend on the interval. However, thanks to the uniform-in-time boundedness of $|p_i|$, we can fix the interval by $[0,P^0_M]$, and this enables us to fix $m_{g'}$ and $M_{g'}$ according to the initial data, namely $M_{G'}$ and $m_{G'}$.
	\item In Proposition \ref{P2.1}, if an ambient space of \eqref{A-3} is one-dimensional $(d=1)$, then the inner product is merely a scalar multiplication, and (3) proves (4) without convexity or concavity assumption on $g$.
	\end{enumerate}
\end{remark}

\subsection{Emergence of asymptotic flocking.}
The CS model is one of the most successful models designing the flocking behavior, and we can still expect the emergent flocking of \eqref{A-3} as well. We recall the definition of (asymptotic) flocking for a model \eqref{A-3}.

\begin{definition}\label{D1.1} Let $(P,Q)$ be a global solution to \eqref{A-3}.
	\begin{enumerate}
	\item We say that $(P,Q)$ exhibits a (mono-cluster) flocking if
	\begin{align*}%\label{flock}
		\sup_{t \geq 0} D_Q(t) < \infty,
		\quad \lim_{t \to \infty} D_P(t) =0. 
	\end{align*}
	\item We say that $(P,Q)$ exhibits a bi-cluster flocking if there exists a nonempty proper subset $S$ of $[N]$ satisfying
	\begin{align*}
		&\sup_{t \geq 0} \max\{ D_{Q,S}(t), D_{Q,[N]-S}(t) \} < \infty,
		\quad \sup_{t \geq 0}\min_{\substack{i \in S \\ j \notin S}}|q_i(t) - q_j(t)| = \infty, \\
		&\lim_{t \to \infty} D_{P,S}(t) = \lim_{t \to \infty} D_{P,[N]-S}(t) =0.
	\end{align*}
	\end{enumerate}
\end{definition}

%We expect a solution to be flocking if kernel $\psi$ has fat tails compared to initial velocity configuration. 

In the following lemma, we estimate the relative distance for an arbitrary collection of agents, which will be used repeatedly through Section \ref{sec:2}, Section \ref{sec:3}, and Section \ref{sec:4}.

\begin{lemma}[Subsystem estimation]\label{L2.1} Let $(P,Q)$ be a solution to \eqref{A-3}. For any $[l] \subset [N]$, we have the following differential inequalities.
\begin{enumerate}
	\item
	\[
	\frac{d}{dt} \| P \|_{[l]} \leq
	-\frac{\kappa\mathcal{M}l}{N}\psi(\|Q\|_{[l]})\|P\|_{[l]} + \frac{2\kappa M_{G'}(N-l)P^0_ML_{\psi, [l]}}{N}\|Q\|_{[l]},
	\]
	\begin{align}\label{Lippsi}
	L_{\psi,[l]}(t) := \sup_{\substack{r,s \geq {q}_{[l]}(t), \\ r \neq s}} \left| \frac{\psi(r)-\psi(s)}{r-s} \right| <\infty,
	\quad {q}_{[l]}(t):= \min_{\substack{i' \in [l] \\ j' \notin [l]}}|q_{i'}(t)-q_{j'}(t)|.
	\end{align}
	\item
	\[
	\frac{d}{dt} \| P \|_{[l]} \leq
	-\frac{\kappa\mathcal{M}l}{N}\psi(\|Q\|_{[l]})\|P\|_{[l]}
	+\frac{4\kappa P^0_M M_{G'} l (N-l)}{N}\max_{\substack{i' \in [l] \\ j' \notin [l]}}{\psi_{i'j'}}.
	\]
\end{enumerate}

\begin{proof}
(1) For simplicity, let $p_{ij}:=p_i-p_j$. We expand $\| P \|_{[l]}$ as
\begin{align*}
\begin{aligned} \label{C-1-1}
\frac{1}{2}\frac{d}{dt} \| P \|^2_{[l]} &= \sum_{i,j \in [l]} p_{ij} \cdot\left(\frac{dp_{i}}{dt} - \frac{dp_j}{dt}\right) \\
&= \frac{\kappa}{N} \sum_{i,j \in [l]} \sum_{k=1}^N
	\underbrace{ p_{ij} \cdot \left(  \psi_{ki} (G(p_k)- G(p_i))-\psi_{kj}(G(p_k)-G(p_j)) \right)}_{=:\mathcal{A}_{ijk}} \\
&= \frac{\kappa}{N}  \sum_{i,j,k \in [l]}  \mathcal{A}_{ijk} + \frac{\kappa}{N} \sum_{i,j \in [l], k \notin [l]} \mathcal{A}_{ijk} \\
& =: \frac{\kappa\mathcal{I}_{1}}{N} + \frac{\kappa\mathcal{I}_{2}}{N}. 
\end{aligned}
\end{align*}

For the estimate of $\mathcal{I}_{1}$, we utilize the symmetry of indices to use the index switching trick $(i,j,k) \to (j,k,i)$ to obtain
\begin{align*}
	\mathcal{I}_{1}
	&= \sum_{i,j,k \in [l]} \psi_{ki} p_{ij} \cdot (G(p_k)- G(p_i))
	- \sum_{i,j,k \in [l]} \psi_{kj} p_{ij} \cdot (G(p_k)-G(p_j)) \\
	&= \sum_{i,j,k \in [l]} \psi_{ki} p_{ij} \cdot (G(p_k)- G(p_i))
	+ \sum_{i,j,k \in [l]} \psi_{ki} p_{jk} \cdot (G(p_k)-G(p_i))  \\
	&= \sum_{i,j,k \in [l]} \psi_{ki} p_{ik} \cdot (G(p_k)- G(p_i)) \\
	&\leq -\mathcal{M}\sum_{i,j,k \in [l]} \psi_{ki}|p_{ik}|^2 \leq -\mathcal{M}l\psi(\|Q\|_{[l]})\|P\|_{[l]}^2.
\end{align*}

For the estimate of $\mathcal{I}_{2}$, we use the Lipschitz continuity of $\psi$ to get
\[
	|\psi_{ki}-\psi_{kj}| \leq L_{\psi,[l]}(t)\big||q_k-q_i|-|q_k-q_j|\big|
	\leq L_{\psi,[l]}(t)|q_i-q_j|,
\]
where $L_{\psi,[l]}$ is a nonnegative function defined as \eqref{Lippsi}. Then a direct computation yields
\begin{align*}
\mathcal{I}_{2} &= \sum_{i,j \in [l], k \notin [l]} p_{ij} \cdot \left(\psi_{ki}(G(p_k)-G(p_i))-\psi_{kj}(G(p_k)-G(p_j))\right) \\
&= \sum_{i,j \in [l], k \notin [l]} p_{ij} \cdot  [\psi_{ki}( G(p_j)- G(p_i)) +(\psi_{ki}-\psi_{kj})(G(p_k)-G(p_j))] \\
& \leq \sum_{i,j \in [l], k \notin [l]} (\psi_{ki}-\psi_{kj})p_{ij} \cdot (G(p_k)-G(p_j))\\
& \leq 2M_{G'}P^0_M L_{\psi, [l]}
 \sum_{i,j \in [l], k \notin [l]} |q_i - q_j | |p_i-p_j|. \\
 & \leq 2M_{G'}(N-l)P^0_M L_{\psi, [l]}\|P\|_{[l]}\|Q\|_{[l]},
\end{align*} 
where we used 
\[
	0 \leq \mathcal{M}|p_i-p_j|^2 \leq p_{ij}(G(p_i)-G(p_j))
\]
for the first inequality.
Combining the estimates altogether, we obtain
\[
	\frac{d}{dt} \| P \|_{[l]} \leq
	-\frac{\kappa\mathcal{M}l}{N}\psi(\|Q\|_{[l]})\|P\|_{[l]} + \frac{2\kappa M_{G'}(N-l)P^0_ML_{\psi, [l]}}{N}\|Q\|_{[l]}.
\]
(2) We estimate $\mathcal{I}_2$ as following.
\begin{align*}
\mathcal{I}_{2}
& \leq \sum_{i,j \in [l], k \notin [l]} (\psi_{ki}-\psi_{kj})p_{ij} \cdot (G(p_k)-G(p_j))\\
& \leq 2\max_{\substack{i' \in [l] \\ j' \notin [l]}}\psi_{i'j'}\sum_{i,j \in [l], k \notin [l]}|p_i-p_j||G(p_k)-G(p_j)| \\
&\leq 4P^0_M M_{G'} l(N-l)\max_{\substack{i' \in [l] \\ j' \notin [l]}}\psi_{i'j'}\|P\|_{[l]}.
\end{align*} 
Together with the estimate of $\mathcal{I}_1$, we have
\[
	\frac{d}{dt} \| P \|_{[l]} \leq
	-\frac{\kappa\mathcal{M}l}{N}\psi(\|Q\|_{[l]})\|P\|_{[l]}
	+\frac{4\kappa P^0_M M_{G'} l(N-l)}{N}\max_{\substack{i' \in [l] \\ j' \notin [l]}}{\psi_{i'j'}}.
\]
\end{proof}

In particular, choice of $[l]=[N]$ leads to the following result on the emergence of flocking.

\end{lemma}

\begin{theorem}[Emergence of asymptotic flocking]\label{T2.1}
Let $(P,Q)$ be a solution to \eqref{A-3}. 
\begin{enumerate}
\item The following three statements are equivalent.
	\begin{enumerate}
		\item $(P,Q)$ exhibits flocking; $\sup_{t \geq 0}D_Q(t) < \infty, ~ \lim_{t \to 0}D_P(t)=0$.
		\item $D_P$ decays exponentially; $D_P(t) \leq Be^{-Ct}, ~ B,C>0$. 
		\item Agents are spatially bounded; $\sup_{t \geq 0}D_Q(t) < \infty$.
	\end{enumerate} \vspace{.2cm}
\item Suppose that
	\[
		\|P^0\| < \frac{\mathcal{M}\kappa}{M_{G'}}\int_{\|Q^0\|}^{+\infty}\psi(s)\,ds.
	\]
	Then $(P,Q)$ exhibits flocking. In particular, if $\|\psi\|_{L^1(\bbr_+)} = \infty$, then flocking happens unconditionally.
\end{enumerate}
\end{theorem}

\begin{proof}
	Define the function $\mathcal{L}$ as
	\begin{equation*}
		\quad \mathcal{L}(t):= \frac{{\mathcal{M}\kappa}}{ M_{G'}} \int_{ \|Q^0\|}^{ \|Q(t)\|}\psi(s)\,ds+\|P(t)\|.
	\end{equation*}
	We claim that $\dot{\mathcal{L}} \leq 0$. We first observe that
	\begin{align}\label{A-8}
		\left| \frac{d}{dt}\| Q \|^2_{[l]} \right|
		= 2 \left| \sum_{i,j \in [l]} ( q_i-q_j ) \cdot ( G(p_i)-G(p_j) ) \right|
		\leq 2M_{G'}\| P \|_{[l]} \| Q \|_{[l]}. 
	\end{align}
	We then choose $[l]=[N]$ and apply Lemma \ref{L2.1} to obtain
	\begin{align}\label{A-9}
	\frac{d}{dt} \| P \| \leq
	-{\kappa\mathcal{M}}\psi(\|Q\|)\|P\|.
	\end{align}
	Then this proves the claim:
	\begin{align*}
	\frac{d\mathcal{L}}{dt} &= \frac{\mathcal{M}\kappa}{M_{G'}} \psi( \|Q(t)\|)\frac{d\|Q(t)\|}{dt}+\frac{d\|P(t)\|}{dt}\\
	&\le {\mathcal{M}\kappa} \psi( \|Q(t)\|)\|P(t)\|-{\mathcal{M}\kappa} \psi( \|Q(t)\|)\|P(t)\|=0.
	\end{align*}
	($\bullet$ Proof of (1)) Implications from (a) to (c) and (b) to (a) are clear. Suppose that (c) holds. Since any norms are equivalent in a finite dimensional space, we prove the statement for the norm $\| \cdot \|$. We have
	\[
		\frac{d}{dt} \| P \| \leq
	-{\kappa\mathcal{M}}\psi(\|Q\|)\|P\| 
	\leq  -{\kappa\mathcal{M}}\psi\left(\sup_{t \geq 0}\|Q(t)\|\right)\|P\| 
	\leq -C\|P\|	
	\]
	for some positive constant $C$ independent of time. This yields
	\[
		\|P(t)\| \leq e^{-tC}\|P^0\|.
	\]
	($\bullet$ Proof of (2)) Since $\mathcal{L}$ decrease in time, we have
	\begin{align}\label{A-10}
		\frac{{\mathcal{M}\kappa}}{ M_{G'}} \int_{ \|Q^0\|}^{ \|Q(t)\|}\psi(s)\,ds+\|P(t)\|
		= \mathcal{L}(t) \leq \mathcal{L}(0) = \|P^0\|
		< \frac{\mathcal{M}\kappa}{M_{G'}}\int_{\|Q^0\|}^{+\infty}\psi(s)\,ds.
	\end{align}
	This proves $\sup_{t \geq 0}\|Q(t)\| < \infty$, and hence $\sup_{t \geq 0}D_Q(t) < \infty$. Therefore flocking emerges.
\end{proof}

In particular, Theorem \ref{T2.1} states that if flocking happens, then relative states converge at an exponential rate. Therefore, if agents are far enough, they will not collide.

\begin{corollary}\label{C2.1}
	Suppose that there exists a positive constant $0<M<\infty$ satisfying
	\begin{align}\label{A-11}
		\frac{M_{G'}\|P^0\|}{\kappa\mathcal{M}} < \min \left\{ \int_{\|Q^0\|}^M \psi(r) dr, \psi(M)\min_{i,j \in [N]}|q_i^0-q_j^0| \right\}.
	\end{align}
	Then we have
	\[
		\inf_{t \geq 0}\min_{i,j \in [N]}|q_i(t)-q_j(t)| > 0.
	\]
\end{corollary}

\begin{proof}
From \eqref{A-10} and \eqref{A-11}, we have $\sup_{t \geq 0}\|Q(t)\| < M < \infty$. This yields
\begin{align*}
	|q_i(t)-q_j(t)| &\geq |q_i^0-q_j^0| - \int_0^t |G(p_i(s))-G(p_j(s))| ds \\
	&\geq |q_i^0-q_j^0| - M_{G'}\int_0^t D_P(s) ds \\
	&\geq |q_i^0-q_j^0| - M_{G'}\int_0^t \|P(s)\| ds \\
	&\geq |q_i^0-q_j^0| - M_{G'}\|P^0\| \int_0^t \exp\left(-\kappa \mathcal{M} \int_0^s \psi(\|Q(\tau)\|)d\tau \right) ds  &&(\because \eqref{A-9}) \\
	&\geq |q_i^0-q_j^0| - M_{G'}\|P^0\| \int_0^t \exp\left(-\kappa \mathcal{M} \int_0^s \psi(M) d\tau \right) ds \\
	&\geq |q_i^0-q_j^0| - \frac{M_{G'}\|P^0\|}{\kappa\mathcal{M}\psi(M)}>0.
\end{align*}
\end{proof}

\begin{remark}
	Suppose that kernel is of the form
	\[
		\psi(|q|)=|q|^{-\alpha}, \quad \alpha>1.
	\]
	In this case, even if $\psi \notin (L^\infty \cap C^{0,1})(\bbr_+;\bbr_+)$, the result in Corollary \ref{C2.1} still holds. In fact, a priori condition relaxes to 
	\[
		\|P^0\| < \frac{\mathcal{M}\kappa}{M_{G'}}\int_{\|Q^0\|}^{+\infty}\psi(s)\,ds.
	\]
	and the proof will be provided in Theorem \ref{T4.1}.

\end{remark}

\subsection{Application to bi-cluster flocking}

Theorem \ref{T2.1} demonstrates a close relationship between spatial boundedness and the emergence of flocking. Likewise, spacial boundedness plays an essential role in bi-cluster flocking. 

\begin{proposition}\label{P2.2} Let $(P,Q)$ be a solution to \eqref{A-3}. Then the following two statements are equivalent.
			\begin{enumerate}
				\item  $(P,Q)$ exhibits bi-cluster flocking.
				\item There exists a partition $\{A,B\}$ of $[N]$ satisfying
				\begin{align*}
					&\sup_{t \geq 0} \max\{ D_{Q,A}(t), D_{Q,B}(t) \} < \infty,
					\quad \sup_{t \geq 0}\min_{\substack{i \in A \\ j \in B}}|q_i(t) - q_j(t)| = \infty.
			\end{align*}
			\end{enumerate}
\end{proposition}

To prove Proposition \ref{P2.2}, we introduce a preliminary lemma.

	\begin{lemma}\label{L2.2}
	Let $(P,Q)$ be a solution to \eqref{A-3}. Suppose that there exists a partition $\{A,B\}$ of $[N]$ satisfying
	\begin{align*}%\label{A-12}
		\sup_{t \geq 0}\max\{ D_{Q,A}(t), D_{Q,B}(t)\} < \infty.
	\end{align*}
	Then whenever two groups generate different clusters, two groups segregate:
	\[
		\sup_{t \in \bbr_+}\min_{i \in A, j \in B}|q_i(t)-q_j(t)|=\infty
		\quad \Longrightarrow \quad
		\lim_{t \to \infty}\min_{i \in A, j \in B}|q_i(t)-q_j(t)|=\infty.
	\]
	\end{lemma}

	\begin{proof}
		It suffices to prove
		\[
			\liminf_{t \to \infty}\min_{i \in A, j \in B}|q_i(t)-q_j(t)|=\infty.
		\]
		Suppose on the contrary that
		\[
			\liminf_{t \to \infty}\min_{i \in A, j \in B}|q_i(t)-q_j(t)|=M_{AB}<\infty.
		\]
		Then there exists a time sequence $\{t_n\}_{n \in \mathbb{N}}$ satisfying
		\[
			t_1<t_2<\cdots, \quad \lim_{n \to \infty}t_n=\infty,
			\quad \sup_{n \in \mathbb{N}}\min_{i \in A, j \in B}|q_i(t_n)-q_j(t_n)| < 1+M_{AB}.
		\] 
		As each groups are spatially bounded, we have
		\[
			\sup_{n \in \mathbb{N}}\max_{i,j \in [N]}|q_i(t_n)-q_j(t_n)| < 1+M_{AB} + \sup_{t \in \bbr_+}\|Q(t)\|_{A} + \sup_{t \in \bbr_+}\|Q(t)\|_{B} =: M'_{AB} < \infty.
		\]
		Then for any $T>0$,
		\[
			\sup_{n \in \mathbb{N}}\sup_{t \in (t_n,t_n+T)}D_{Q}(t) \leq M'_{AB} + 2TM_{G'}P^0_M =: M'_{AB,T}< \infty,
		\]
		since $p_i(t)$ (resp. $\dot{q}_i(t)=(G(p_i(t))$) is bounded above by $P^0_M$ (resp. $M_{G'}P^0_M$) from Proposition \ref{P2.1}. Therefore if $(t_1,\infty) \subset \cup_{n \in \mathbb{N}} (t_n,t_n+T)$ for some finite $T$, flocking must emerge from Theorem \ref{T2.1}. Since $A$ and $B$ generates different cluster, this cannot happen, which yields
		\[
			\limsup_{n \to \infty}(t_{n+1}-t_n)=\infty.
		\]
		Passing to a subsequence, we may assume $\lim_{n \to \infty}(t_{n+1}-t_n)=\infty$. Let $\ell$ be an index maximizing $|p_i|$. We recall that
		\begin{align}\label{A-12-1}
		\frac{d}{dt}|p_\ell|^2
		= \frac{2\kappa}{N} \sum_{k=1}^N \psi_{k\ell} p_\ell \cdot (G(p_k)-G(p_\ell))
		\leq \frac{2\kappa}{N} \psi(D_Q)  \sum_{k=1}^N p_\ell \cdot  (G(p_k)-G(p_\ell)).
		\end{align}
		As a maximum of R.H.S. in \eqref{A-12-1} is achieved when each $p_i$ have same direction (i.e. $\cos(p_i,p_j)=1)$, we further estimate
		\begin{align*}
		\frac{d}{dt}|p_\ell|^2
		&\leq \frac{2\kappa}{N} \psi(D_Q)  \sum_{k=1}^N p_\ell \cdot \left(\frac{g(|p_k|)}{|p_\ell|}p_\ell - \frac{g(|p_\ell|)}{|p_\ell|}p_\ell\right) \\
		&\leq \frac{2\kappa m_{G'}}{N} \psi(D_Q)  \sum_{k=1}^N |p_\ell|\left({|p_k|} - {|p_\ell|}\right) \\
		&= - \frac{2\kappa m_{G'}}{N} \psi(D_Q) |p_\ell|\left( N|p_\ell|-\sum_{k=1}^N|p_k| \right)\\
		&= - {2\kappa m_{G'}}\psi(D_Q) |p_\ell|\left( |p_\ell|-\frac{1}{N}\left|\sum_{k=1}^N p_k \right| \right) && (\because \cos(p_i,p_j)=1) \\
		&=: - {2\kappa m_{G'}}\psi(D_Q) |p_\ell|\left( |p_\ell|-{|p^0_{\mathrm{ave}}|} \right).
		\end{align*}
	Note that $p^0_{\mathrm{ave}}$ is a constant vector, since the derivative of $\sum_{k=1}^N p_k$ is zero. Thus we have
	\[
		\frac{d}{dt}\left( |p_\ell|-{|p^0_{\mathrm{ave}}|} \right)
		= \frac{d}{dt}|p_\ell|
		\leq - {\kappa m_{G'}}\psi(D_Q)\left( |p_\ell|-{|p^0_{\mathrm{ave}}|} \right),
	\]
	which leads to
	\[
		0 \leq |p_\ell(t)| - |p^0_\mathrm{ave}|\leq  \exp\left( -\kappa m_{G'}\int_s^t \psi(D_Q(u)) du \right)(|p_\ell(s)|-|p^0_\mathrm{ave}|), \quad s \leq t,
	\]
	where the first inequality comes from the maximality of $\ell$. Now fix $T$ and take $N \gg 1$, so that each interval $(t_n, t_n+T)$ is disjoint for each $n \geq N$. Then 
	\[
		0 \leq |p_\ell(t_n+T)| - |p^0_\mathrm{ave}|
		\leq  \exp\left( -\kappa m_{G'}T \psi(M'_{AB,T}) \right)(|p_\ell(t_n)|-|p^0_\mathrm{ave}|), \quad n \geq N.
	\]
	On the other hand, $|p_\ell|$ is unconditionally decreasing from Proposition \ref{P2.1}. Thus,
	\begin{align*}
		0 \leq |p_\ell(t_{n+1}+T)| -  |p^0_\mathrm{ave}|
		&\leq  \exp\left( -\kappa m_{G'}T \psi(M'_{AB,T}) \right)(|p_\ell(t_{n+1})|-|p^0_\mathrm{ave}|) \\
		&\leq  \exp\left( -\kappa m_{G'}T \psi(M'_{AB,T}) \right)(|p_\ell(t_{n}+T)|-|p^0_\mathrm{ave}|) \\
		&\leq \exp\left( -\kappa m_{G'}2T \psi(M'_{AB,T}) \right)(|p_\ell(t_{n})|-|p^0_\mathrm{ave}|).
	\end{align*}
	Then a straightforward induction yields
	\[
		0
		\leq |p_\ell(t^*)| - |p^0_\mathrm{ave}|
		\leq \exp\left( -\kappa m_{G'}(m+1)T \psi(M'_{AB,T}) \right)(|p_\ell(t_{n})|-|p^0_\mathrm{ave}|),
		\quad t^* \geq t_{n+m}+T,
	\]
	and therefore
	\[
		\lim_{t \to \infty}|p_\ell(t)| = |p^0_\mathrm{ave}|.
	\]
	We again use the maximality of $\ell$ and apply the squeeze theorem to find
	\[
		N|p^0_\mathrm{ave}|
		=\left| \sum_{k=1}^N p_k(t) \right|
		\leq \sum_{k=1}^N  \left| p_k(t) \right| 
		\leq N\left| p_\ell(t) \right|,
		~ \text{so that} ~ \lim_{t \to \infty} \sum_{k=1}^N  \frac{1}{N} \left| p_k(t) \right| = |p^0_\mathrm{ave}|.
	\]
	Now we claim
	\begin{align}\label{A-13}
		\lim_{t \to \infty}p_k(t) = p^0_\mathrm{ave}, \quad k \in [N].
	\end{align}
	It suffices to show $\lim_{t \to \infty} \left| p_k(t) \right| = |p^0_\mathrm{ave}|$ for each $k$. Suppose the contrary. Then since
	\[
		\limsup_{t \to \infty} |p_k(t)| \leq \limsup_{t \to \infty} |p_\ell(t)| = |p^0_\mathrm{ave}|, \quad k \in [N],
	\] 
	there exists a constant $P_m$ satisfying
	\[
		\liminf_{t \to \infty}\min_{i \in [N]}|p_i(t)| < P_m < |p^0_\mathrm{ave}|,
	\]
	and there exists a time sequence $\{s_n\}_{n \in \mathbb{N}}$ such that
	\[
		0<s_1<s_2<\cdots, \quad \lim_{n \to \infty}s_n = \infty, \quad \min_{i \in [N]}|p_i(s_n)| < P_m.
	\]
	This leads to
	\[
		|p_\mathrm{ave}^0| = \frac{1}{N}\sum_{i=1}^N \lim_{n \to \infty}|p_i(s_n)| \leq \frac{1}{N}P_m + \lim_{n \to \infty}\frac{N-1}{N}|p_\ell(s_n)|
		< |p_\mathrm{ave}^0| %= \frac{1}{N}\sum_{i=1}^N \lim_{n \to \infty}|p_i(s_n)|,
	\]
	which yields a contradiction, verifying the claim \eqref{A-13}. Finally, since flocking does not happen, from Theorem \ref{T2.1} we have
	\[
		\|P(t_n)\| \geq \frac{\mathcal{M}\kappa}{M_{G'}}\int_{\|Q(t_n)\|}^{+\infty}\psi(s)\,ds
	\]
	for any $t_n$. As $\|Q(t_n)\|$ is uniformly bounded in $n$ from the definition of $t_n$, let $Q_M< \infty$ be its upper bound. Since $\|P(t_n)\|$ converges to zero from \eqref{A-13}, we obtain
	\[
		0 < \frac{\mathcal{M}\kappa}{M_{G'}}\int_{Q_M}^{+\infty}\psi(s)\,ds
		\leq \frac{\mathcal{M}\kappa}{M_{G'}}\int_{\|Q(t_n)\|}^{+\infty}\psi(s)\,ds
		\leq \|P(t_n)\| \xrightarrow{n \to \infty} 0,
	\]
	and this completes the proof by contradiction.
	\end{proof}

\begin{proof}[Proof of Proposition \ref{P2.2}]
	Clearly, (1) implies (2). Suppose that (2) holds. From Lemma \ref{L2.1}, we have
	\[
		\frac{d}{dt} \| P \|_A \leq
		-\frac{\kappa\mathcal{M}l}{N}\psi(\|Q\|_A)\|P\|_A
		+\frac{4\kappa P^0_M M_{G'} l (N-l)}{N}\max_{\substack{i' \in A \\ j' \in B}}{\psi_{i'j'}}.
	\]
	From the assumptions on (2) and Lemma \ref{L2.2}, we have
	\[
		\inf_{t \in \bbr_+}\left( \frac{\kappa\mathcal{M}l}{N}\psi(\|Q(t)\|_A) \right) \geq C_1 > 0,
		\quad \lim_{t \to \infty}\max_{\substack{i' \in A \\ j' \in B}}{\psi_{i'j'}} = 0,
	\]
	for some positive constant $C_1$. Therefore for any $\varepsilon > 0$, there exists a time $T$ satisfying
	\[
		\frac{d}{dt} \| P(t) \|_A \leq
		-C_1\|P(t)\|_A + \varepsilon, \quad \forall t > T(\varepsilon) > 0.
	\]
	Then by the comparison principle, we have
	\[
		0 \leq \limsup_{t \to \infty}\|P\|_A \leq \frac{\varepsilon}{C_1}.
	\]
	Since the choice of $\varepsilon > 0$ is arbitrary, we conclude $\|P(t)\|_A \xrightarrow{t \to \infty} 0$. The proof of $\|P(t)\|_B \xrightarrow{t \to \infty} 0$ is similar.
\end{proof}

	\begin{example} 
	
	In this example, we briefly sketch an example that achieves a bi-cluster flocking. For convenience, let $G=\mathrm{Id}$ so that $m_{G'}=M_{G'}=\mathcal{M}=1$. Suppose that $D_Q^0 \neq 0$. Then from \cite[Theorem 5.1]{CHHJK}, under some well-prepared initial configuration, there exists a set of indices $[l]$ (after reordering), which is a nonempty proper subset of $[N]$, and a positive constant $C$ such that
	\[
		\min_{i \in [l] ~ j \notin [l]}|q_i(t)-q_j(t)| \geq Ct.
	\]
	This leads to
	\[
		\int_0^t \max_{i \in [l], j \notin [l]}\psi(q_{i}(s)-q_{j}(s)) ds
		\leq \int_0^t \psi(Cs) ds \leq \frac{\|\psi\|_{L^1(\bbr_+)}}{C}.
	\]
	Then, integrating the second estimate in Lemma \ref{L2.1} leads to
	\[
		\|P(t)\|_{[l]} + \frac{\kappa l}{N}\int_{\|Q^0\|_{[l]}}^{\|Q(t)\|_{[l]}}\psi(r)dr
		\leq \|P^0\|_{[l]} + \frac{4\kappa P^0_M l (N-l)\|\psi\|_{L^1(\bbr_+)}}{CN}.
	\]
	Therefore, if a velocity deviation in a group $[l]$ is small and $C$ is sufficiently large in the sense that
	\[
		\frac{N}{\kappa l}\|P^0\|_{[l]} + \frac{4 P^0_M (N-l)\|\psi\|_{L^1(\bbr_+)}}{C} < \int_{\|Q^0\|_{[l]}}^{\infty}\psi(r)dr,
	\]
	then we have $\sup_{t \in \bbr_+}\|Q(t)\|_{[l]} < \infty$. Similarly, we have $\sup_{t \in \bbr_+}\|Q(t)\|_{[N]-[l]} < \infty$ for large $C>0$, and this implies bi-cluster flocking. Note that $C$ can be chosen sufficiently large for a suitable choice of initial data. For the detail, we refer to \cite{CHHJK}.

	\end{example}

%\begin{remark}
%	Galliean inv leads to loss of generality.
%\end{remark}

\section{Analysis under weakly singular communications on the real line}\label{sec:3} 
\setcounter{equation}{0}

In this section, we consider the kernel of form
\[
	\psi(x)=\frac{1}{|x|^\alpha}, \quad \alpha \in (0,1).
\]
In this case, particles may collide (for the colliding example, see \cite{BHK}) and the vector field blows up. The description of such a solution is not straightforward, as provided in the following Definition and Theorem  \cite{P15,P14}.

\begin{definition}\label{D3.1}
Let $B_i$ and $\psi_n$ be defined as
\[
	B_i(t) := \{ k \in [N] : x_k(t) \neq x_i(t) ~ \text{or} ~ v_k(t) \neq v_i(t) \},
\]
\[
	\psi_n(s) :=
	\begin{cases}
		\psi(s) \quad &\mathrm{if} \quad s \geq (n-1)^{-\frac{1}{\alpha}} \\
		\mathrm{smooth ~ and ~ monotone} \quad &\mathrm{if} \quad n^{-\frac{1}{\alpha}} \leq s \leq (n-1)^{-\frac{1}{\alpha}} \\
		n \quad &\mathrm{if} \quad s \leq n^{-\frac{1}{\alpha}}
	\end{cases}
\]
and let $0=T_0 \leq T_1 \leq T_{N_s}$ be the set of all times of sticking (i.e. $x_i(t)-x_j(t)=v_i(t)-v_j(t)$ for some $i$) and $T_{N_s+1}:=T$ be a given positive number. For $n \in \{0,\cdots,N_s\}$, on each interval $[T_n,T_{n+1}]$, consider the problem
\begin{align}\label{B-1}
\begin{cases}
	\dot{x}_i = v_i, \\
	\dot{v}_i = \frac{1}{N}\sum_{k \in B_i(T_n)}(v_k-v_i)\psi_n(|x_k-x_i), \\
	x_i \equiv x_j \quad \mathrm{if} \quad j \notin B_i(T_n),
\end{cases}
\end{align}
for $t \in [T_n,T_{n+1}]$, with initial data $x(T_n),v(T_n)$. We say that $(x,v)$ solve \eqref{B-1} on the time interval $[0,T]$ with weight $\psi(s)=s^{-\alpha}$ if and only if for all $n=0,\cdots,N_s$ and arbitrary small $\varepsilon >0$, the function $x \in (C^1([0,T]))^{Nd}$ is a weak in $(W^{2,1}([T_n,T_{n+1}-\varepsilon]))^{Nd}$ solution of \eqref{B-1}.
\end{definition}

\begin{proposition}\label{P3.1} ~
\begin{enumerate}
\item Let $\alpha \in (0, \frac{1}{2})$ be given. Then for all $T>0$ and arbitrary initial data, there exists a unique $x \in W^{2,1}([0,T]) \subset C^1([0,T])$ that solves \eqref{A-3} with communication weight $\psi(s)=\frac{1}{|x|^\alpha}$ weakly in $W^{2,1}([0,T])$.
\item Let $\alpha \in (\frac{1}{2},1)$ be given. Then there exists a unique solution in the sense of Definition \ref{D3.1}.
\end{enumerate}
\end{proposition}

Although the description of a collisional solution under a singular kernel is somewhat non-trivial, if we restrict \eqref{A-3} on the real line, we may convert it into the first-order model, and its analysis may hint at the property of a solution in the second-order model as well. In this section, we are interested in the model \eqref{A-3} on the \emph{real line}, equipped with a weakly singular kernel:
\begin{align}\label{B-2}
	\begin{cases}
		\dot{q}_i= G(p_i),\quad t>0, \quad i \in [N],\\
		\displaystyle \dot{p}_i=\frac{\kappa}{N}\sum_{k=1}^N\psi(q_k-q_i)(G(p_k)-G(p_i)) \vspace{.2cm},\\
		(q_i,p_i) \big|_{t = 0+} = (q_i^0,p_i^0), \quad p_i,q_i \in \bbr,
	\end{cases}
	\quad \psi(s)=\frac{1}{|s|^\alpha}, \quad \alpha \in (0,1).
\end{align}

If $\psi$ is regular and $(P,Q)$ is a classical solution of \eqref{B-2}, then we have a following relation:
	\[
		\int_0^t \psi(q_k(s)-q_i(s))(G(p_k(s))-G(p_i(s))) ds = \int_0^{q_k(t)-q_i(t)}\psi(s) ds.
	\]
Therefore $P$ is a solution of 
\begin{align}\label{B-3}
	\dot{q}_i = G(\nu_i + \frac{\kappa}{N}\sum_{k=1}^N \Psi(q_k-q_i)),
\end{align}
provided that two systems are coupled by the following relationship:
\begin{align}\label{B-4}
	\quad \Psi(r) := \int_0^r \psi(x) dx,
	\quad \nu_i := p_i^0 - \frac{\kappa}{N}\sum_{k=1}^N \Psi(q_k^0-q_i^0).
\end{align}
On the other hand, the converse holds; if $\Psi$ is differentiable, a solution of \eqref{B-3} is also a solution of \eqref{B-2} under \eqref{B-4}, and therefore two models are equivalent. What if $\psi$ weakly singular? In this case, we have
\begin{align}\label{B-4-1}
	\psi(q) = \frac{1}{|q|^\alpha} ~ (\alpha \in (0,1)) \quad \Longrightarrow \quad
	\Psi(q)=\int_0^q \psi(r)dr=\text{sgn}(q)\frac{|q|^{1-\alpha}}{1-\alpha}.
\end{align}
As $\Psi$ is continuous, Peano’s theorem guarantees a classical solution of \eqref{B-3}. However, this may \emph{not} be a classical solution of \eqref{B-2}, since a solution of \eqref{B-2} requires more regularity of $q_i$ than \eqref{B-3}, but the regularity of $\Psi$ (and hence regularity of $\dot{q}_i$) breaks down at the origin.

\begin{example}\label{E3.1}
	Consider a two-particle system with 
	\[
	G=\mathrm{Id}, \quad \psi(q)=\frac{1}{\sqrt{q}}, \quad \Psi(q) = 2\mathrm{sgn}(q)\sqrt{|q|}.
	\]
	Then, \eqref{B-3} is of the form
	\[
		\dot{q_1} = \nu_1 + \frac{\kappa}{2}\Psi(q_2-q_1),
		\quad \dot{q_2} = \nu_2 + \frac{\kappa}{2}\Psi(q_1-q_2).
	\]
	If we pose $\nu_1=\nu_2 = 0$ and $q_1^0 \geq q_2^0$, we have a following classical solution of \eqref{B-3}. 
	\[
		q_i =
		\begin{cases}
		\displaystyle \frac{1}{2}\left((q_1^0+q_2^0)+ \kappa^2\big(t-\frac{1}{\kappa}\sqrt{|q_1^0-q_2^0|}\big)^2\right), \quad &\text{\rm{if} $i=1$, $t < \displaystyle\frac{1}{\kappa}\sqrt{|q_1^0-q_2^0|}$,} \vspace{.2cm}\\
		\displaystyle \frac{1}{2}\left((q_1^0+q_2^0)- \kappa^2\big(t-\frac{1}{\kappa}\sqrt{|q_1^0-q_2^0|}\big)^2\right), \quad &\text{\rm{if} $i=2$, $t < \displaystyle\frac{1}{\kappa}\sqrt{|q_1^0-q_2^0|}$,} \vspace{.2cm} \\
		\displaystyle \frac{q_1^0+q_2^0}{2}, \quad &\text{\rm{if} $i=1,2$, $t \geq \displaystyle\frac{1}{\kappa}\sqrt{|q_1^0-q_2^0|}$.}
		\end{cases}
	\]
	However, since $\dot{q_i}$ is not differentiable, we cannot recover a classical solution of \eqref{B-2}.
	%Furthermore, we have
	%\[
	%	\dddot{q}_1(t)=-\kappa^2\delta\left(t-\frac{1}{\kappa}\sqrt{|q_1^0-q_2^0|}\right),
	%	\quad \dddot{q}_2(t)=\kappa^2\delta\left(t-\frac{1}{\kappa}\sqrt{|q_1^0-q_2^0|}\right)
	%\]
\end{example}

The above example illustrates that two models are not equivalent under the classical regime if $\psi$ is singular; a solution of \eqref{B-3} need not be twice differentiable. Therefore, if one attempts to make two models equivalent, one needs to enlarge the concept of solution of \eqref{B-2}. It turns out Sobolev space $W^{2,1}$ is an appropriate function space, as described in the following theorem. 

\begin{theorem}\label{T3.1} Let $(P,Q)$ be a solution to \eqref{B-2}. Then the following assertions holds.
	\begin{enumerate}
		\item The model \eqref{B-2} has a unique global (weak) solution where $q_i \in W^{2,\gamma}([0,T])$ for each $i \in [N]$, $T \in \bbr_+$, and
		\[
			\gamma \in \bigg[1, \frac{1}{\max\{ 1-K, \alpha\}}\bigg), \quad 
			K := \frac{m_{G'}2^{1-2\alpha}(1-\alpha)}{NM_{G'}\alpha}.
		\]
	\item Flocking emerges unconditionally:
	\[
		\sup_{t \geq 0}\max_{i,j \in [N]}|q_i(t)-q_j(t)| < \infty,
		\quad
		\max_{i,j \in [N]}|p_i(t)-p_j(t)| \lesssim e^{-Ct}, \quad C>0.
	\]
	\end{enumerate}
\end{theorem} \vspace{.5cm}

%\begin{theorem}\label{T3.1} Suppose that two models \eqref{B-2} and $\eqref{B-3}$ are related as $\eqref{B-4}$. Then two models are equivalent and exhibits the flocking in the following sense.
%	\begin{enumerate}
%		\item The models \eqref{B-3} and \eqref{B-4} has a unique global (weak) solution, where $q_i \in W^{2,\gamma}([0,T])$ for each $i \in [N]$, $T \in \bbr_+$, and
%		\[
%			\gamma \in \Bigg[1, \frac{1}{1-\min\{ 1-\alpha, 2 \kappa/N \}}\ \Bigg).
%		\]
%		Furthermore, we have $(p_i - \sum_{i=k}^N \frac{p_k^0}{N})=(G^{-1}(\dot{q}_i) - \sum_{i=k}^N \frac{\nu_k}{N}) \in W^{1,\gamma}(\bbr_+)$ for each $i$ and $\gamma$.
%	\item Let $(Q^1,P^1)$ and $Q^2$ be a solution of \eqref{B-2} and \eqref{B-3} with initial data $(Q^0,P^0)$ and $(Q^0,\mathcal{N})$, respectively. Then we have $Q^1 \equiv Q^2$. 
%	\item Agents aggregate unconditionally with exponential rate, even if collisions or sticking\footnote{Their precise meanings will be provided soon in Definition \ref{D3.2}.} happen. More precisely, we have:
%	\[
%		\sup_{t \geq 0}\max_{i,j \in [N]}|p_i(t)-p_j(t)| < \infty,
%		\quad
%		\max_{i,j \in [N]}|\dot{p}_i(t)-\dot{p}_j(t)| \lesssim e^{-Ct}, \quad C>0.
%	\]
%	\end{enumerate}
%\end{theorem}

\begin{remark}\label{R3.1} Below, we list some comments about Theorem \ref{T3.1}.
\begin{enumerate}
	\item Theorem \ref{T3.1} states that $q_i$ is always continuously differentiable and $p_i$ is always differentiable almost everywhere, and the regularity of $p_i$ improves as $\alpha$ decreases. Roughly speaking, when $\alpha$ is close to 1, then $p_i$ is close to an absolutely continuous function, and when $\alpha$ is close to 0, then $p_i$ is close a to Lipschitz continuous function. In fact, if only sticking happens and collision does not occur (see Definition \ref{D3.2}), then $p_i$ can be indeed Lipschitz, as described in Example \ref{E3.1}.
	\item	Equivalence between \eqref{B-2} and \eqref{B-3} is not trivial for a singular kernel. When the kernel $\psi$ is regular, equivalence is essentially based on the following change of variable formula:
	\[
		\int_0^t \psi(q_k(s)-q_i(s))(G(p_k(s))-G(p_i(s))) ds = \int_0^{q_k(t)-q_i(t)}\psi(s) ds.
	\]
The above formula holds if $\psi$ is continuous and $t \mapsto q_k(t)-q_i(t)$ is continuously differentiable. However, if $\psi$ is merely a nonnegative measurable function, even if $t \mapsto q_k(t)-q_i(t)$ is absolutely continuous, change of variable formula requires either monotonicity of $q_k-q_i$ or integrability of $\psi$ and $\psi(q_k(s)-q_i(s))(G(p_k(s))-G(p_i(s)))$ (see Lemma \ref{L3.2}). Therefore, due to the possibility of pathological behavior near a collision time, a change of variable formula cannot be applied directly.
	\item	In \cite{MP18}, the authors provided a framework to rigorously derive a kinetic description of the model \eqref{A-3} (under $G=\mathrm{Id}$) with a weakly singular communication in a weak-atomic sense. For this, the solution should have a regularity of $W^{2,1}$ and therefore the derivation was limited to the case where $\alpha \in (0,1/2)$ (see Proposition \eqref{P3.1}). Theorem \ref{T3.1} states that a weak-atomic solution can be derived for any $\alpha \in (0,1)$ on the real line.
\end{enumerate}
\end{remark}

Before we establish the equivalence between \eqref{B-2} and \eqref{B-3}, we first review the dynamics of \eqref{B-3}. Recall that
 \[
	 \mathcal{N}:=(\nu_1,\nu_2,\cdots,\nu_N), 
	 \quad \Psi(r) := \int_0^r \psi(x) dx,
	\quad \nu_i := p_i^0 - \frac{\kappa}{N}\sum_{k=1}^N \Psi(q_k^0-q_i^0).
\]

\begin{proposition}\label{P3.2}\cite{BHK}
Suppose that communication weight has weak singularity of the following form:
%\[
%	\Psi ~ \text{is increasing and absolutely continuous} \quad \text{ and } \quad \frac{1}{\Psi} \in L^1_{\mathrm{loc}}(\bbr), 
%\]
%which is valid if communication weight has weak singularity:
\[ \psi(q)=\frac{1}{|q|^\alpha}, \quad  0 < \alpha < 1,\quad q \neq 0, \]
and let $Q$ be a solution to \eqref{B-3} with initial data $(Q^0,\mathcal{N})$. For fixed indices $i$ and $j ~ (i \neq j)$, suppose that 
 \[ q_{i}^0>q_j^0. \]
 Then the following trichotomy holds.
	\begin{enumerate}
	\item If $\nu_i>\nu_j$, then $q_i$ and $q_j$ will not collide in finite time:
	\[ q_i(t)>q_j(t) \quad \mbox{for all $~~t\ge0$}. \]
	\item If $\nu_i<\nu_j$, then $q_i$ and $q_j$ will collide exactly once, i.e., there exists a time $t^*$ such that 
	\[ q_i(t)>q_j(t) \quad \mbox{for $~~0\le t < t^*$}, \quad  q_i(t^*)=q_j(t^*) \quad \mbox{and} \quad  q_i(t)<q_j(t) \quad \mbox{for $t>t^*$}. \]
	\item If $\nu_i=\nu_j$, then $q_i$ and $q_j$ will collide in finite time,  and two particles will stick together after their first collision. 
	\item If $\nu_i \neq \nu_j$, then we have
	\[
		\liminf_{t \to \infty}|q_i(t)-q_j(t)| > 0.
	\] 
	\end{enumerate}
\end{proposition}

Therefore, two particles $p_i$ and $p_j$ will overlap in some time unless $(p_i^0-p_j^0)(\nu_i-\nu_j)>0$, and will eventually `stick' if and only if $\nu_i=\nu_j$. It is easy to see that
\[
	\nu_i=\nu_j
	\quad \Longleftrightarrow \quad
	p_i(t)=p_j(t) \quad \text{implies} \quad \dot{p}_i(t) = \dot{p}_j(t),
\]
as illustrated in Example \ref{E3.1}. If a kernel is regular, then (1),(2), and (4) in Proposition \ref{P3.2} still hold, but (3) does not happen; agents never stick unless they are sticking at the initial state \cite{BHK,HKPZ19}. Thus such `finite-in-time sticking' characterizes a singular kernel \cite{P15,P14}. In what follows, we clarify the definition of sticking and related concepts:

\begin{definition}\label{D3.2}
	Let $Q$ be a classical solution of \eqref{B-3}. Consider two agents $q_i$ and $q_j$.
	\begin{enumerate}
		\item We say $q_i$ and $q_j$ collide at time $t$ if
		\[
			q_i(t)=q_j(t) \quad \text{but} \quad \dot{q}_i(t) \neq \dot{q}_j(t).
		\]
		\item We say $q_i$ and $q_j$ stick at time $t$ if
		\[
			q_i(t)=q_j(t) \quad \text{and} \quad \dot{q}_i(t) = \dot{q}_j(t).
		\]
	\end{enumerate}	
\end{definition}

From Proposition \ref{P3.2}, if particles stick at some instance, they stick afterwards. Therefore, if some agents stick at time $t$ among $N$ agents, the system immediately changes into a system of weighted $N'(<N)$ agents. To describe this phenomena, we define sets of collisional indices, sticking indices and their time set as follows:
 \begin{align}\begin{aligned}\label{B-5}
 	&C_i(t):=\{ j \in [N] \mid q_i ~ \text{and} ~ q_j ~ \text{collide at time} ~ t \},  \\
	&S_i(t):=\{ j \in [N] \mid q_i ~ \text{and} ~ q_j ~ \text{stick at time} ~ t \}, \\
	&\mathcal{T} := \bigcup_{i \in [N]} \left( \{ t \in \bbr_+ :  |C_i(t)| \neq 0\} \cup \{ t \in \bbr_+ :  |S_i| ~ \text{is discontinuous at} ~ t \} \right).
\end{aligned}\end{align}
 Note that $C_i$,$S_i$ and $\mathcal{T}$ depends on solution of \eqref{B-3}, and $S_i$ is discontinuous at the instance when two particles start to stick. 
 
 \begin{proposition}\label{P3.3} Suppose that an initial data $(Q^0,\mathcal{N})$ of \eqref{B-3} is given.
 	\begin{enumerate}
	\item System \eqref{B-3} has a unique global classical (i.e. $q_i(t) \in C^1(\bbr_+;\bbr)$ for each $i$) solution. In particular, $C_i(t)$,$S_i(t)$ and $\mathcal{T}$ are well defined for each $i \in [N]$ and $t \in \bbr_+$. 
	\item For each $i \in [N]$, $|C_i(t)|$ is zero for all but finitely many $t \in \bbr_+$.
	\item For each $i \in [N]$ and $0 \leq s \leq t < \infty$, we have $S_i(s) \subset S_i(t)$. In particular, $|S_i|$ is a right-continuous increasing step function.
	\item $\mathcal{T}$ has a finite cardinality.
	\end{enumerate}
 \end{proposition}

\begin{proof} If we prove (1), then the other statements are direct consequences of Proposition \ref{P3.2} and (1). Therefore we focus on the proof of (1). The existence of a global classical solution is guaranteed by Peano's Theorem. Therefore it suffices to verify the uniqueness. Suppose that there exists two solutions $Q=(q_1,\cdots,q_N)$ and $Q'=(q_1',\cdots, q'_2)$ with same initial data $(Q^0,\mathcal{N})=(q_1^0,\cdots,q_N^0,\nu_1,\cdots,\nu_N)$, which is neither collisional nor sticking (i.e. $\prod_{\substack{i,j \in [N] \\ i \neq j}}(q_i^0-q_j^0) \neq 0$). Let $(C_i,S_i,\mathcal{T})$ and $(C_i',S_i',\mathcal{T}')$ be defined as \eqref{B-5} with respect to $Q$ and $Q'$, respectively. From Proposition \ref{P3.2}, there exists finite number of times $\{t_i\}_{i \in [M]}$ sayisfting
\begin{align*}
	[0,\infty) = \bigcup_{c=0}^M [t_c,t_{c+1}), \quad 0=t_0 < t_1 < \cdots <t_M=+\infty, \quad M<\infty,
\end{align*}
	such that $\{t_1,\cdots,t_{M-1}\}=\mathcal{T}$. %Then, since any two agents never collide twice and does not separate once stick, for each $c \in [M]$, there exists indices $i=i(c)$ and $j=j(c)$ such that
%	\[
%	q_i(t)
%	\begin{cases}
%		\neq q_j(t) \quad \text{for} ~ t \in [0,t_c), \\
%		= q_j(t) \quad \text{for} ~ t = t_c.
%	\end{cases}
%	\]
	Similarly, we set $[0,\infty) = \bigcup_{c=0}^{M'} [t'_c,t'_{c+1})$ with respect to $Q'$. Now we use an induction argument to prove $t_c=t_c'$ and $Q=Q'$ on $[0,t_c)$ for each $c=1,2,\cdots,M$. \newline

	$\bullet$ ($c$=1) Since $\Psi$ is locally Lipschitz except the origin and $\Psi$ is not evaluated at $0$ in $t \in [0, \min\{t_1,t_1'\})$, the standard theory of ODE guarantees that $Q(t)=Q'(t)$ in $[0,\min\{t_1,t_1'\})$. Without loss of generality, suppose that $t_1 \leq t_1'$. Since we have global existence of a classical solution, both of $Q$ and $Q'$ uniquely extends to $[0,t_1]$ and they are same. In particular, if $q_i$ and $q_j$ collide or starts to stick at $t_1$, then so are $q_i'$ and $q_j'$. Therefore we have $t_1=t_1'$ and $Q=Q'$ in $[0,t_1]=[0,t_1']$. \newline
	
	$\bullet$ (Inductive step) Suppose $t_n=t'_n$ for $n=1,2,\cdots,c$ and assume that solution is unique in $[0,t_c)$, so that $S_k(t)$ and $C_k(t)$ are well defined for each $k \in [N]$ in $t \in [0,t_c)$. We claim that
	\begin{center}
	for any $t^* \in [t_c, \min\{t_{c+1},t'_{c+1}\})$, we have $Q=Q'$ in $[0,t^*]$.
	\end{center}
	To prove this by contradiction, suppose that
	\begin{align}\label{B-6}
		Q(T) \neq Q'(T) \quad \text{for some} ~ T \in (t_c,t^*].
	\end{align}
	Let $D(t):=\max_{i \in [N]}|q_i(t)-q'_i(t)|$ and $M=M(t)$ be a time-dependent index satisfying $D(t)=|q_M(t)-q'_M(t)|$. Let time $t \in (t_c,t^*]$ and index $M(t)=\ell$ be fixed. If we assume, without loss of generality, that $q_\ell(t) \geq q'_\ell(t)$, then we have
	\begin{align*}
		\dot{q}_\ell(t) - \dot{q}'_\ell(t)
		&= G(\nu_i + \frac{\kappa}{N}\sum_{k=1}^N \Psi(q_k(t)-q_\ell(t))) - G(\nu_i + \frac{\kappa}{N}\sum_{k=1}^N \Psi(q'_k(t)-q'_\ell(t))) \\
		&= \frac{\kappa G'(\tilde{q}_i)}{N}\sum_{k=1}^N
	\left( \Psi(q_k(t)-q_\ell(t)) - \Psi(q'_k(t) - q'_\ell(t)) \right), \quad t \in (t_c + \varepsilon, t^*],
	\end{align*}
where we used the mean value theorem for the last equality. Then definition of $M$ yields
\begin{center}
	$q_k(t)-q'_k(t) \leq q_\ell(t)-q'_\ell(t) ~ \Longleftrightarrow ~ q_k(t) - q_\ell(t) \leq q'_k(t) - q'_\ell(t)$.
\end{center}
As $\Psi$ is increasing, we have $\dot{q}_\ell(t) - \dot{q}'_\ell(t) \leq 0$. From the existence of a global solution, each $\dot{q}_i$ and $\dot{q}'_i$ are uniformly bounded in $[t_c,t^*)$. Therefore we have $\dot{D}(t) \leq 0$ for almost every $t \in [t_c,t^*)$, and
\[
	0 \leq {D}(t^*) \leq {D}(t_{c})=0,
\]
where the equality comes from the induction hypothesis. Therefore, by the same argument as in the $c=1$ case, we have $t_{c+1}=t'_{c+1}$ and $D(t) \equiv 0$ on $(t_{c},t_{c+1}]$. By induction, we conclude $D \equiv 0$ on $(t_c,t^*]$. This contradicts \eqref{B-6}, which completes the proof for not overlapping initial data. The proof for a collisional or sticking initial data follows by letting $t_1=0$.
\end{proof}

	\begin{lemma}\label{L3.1} Let $(P,\mathcal{N})$ be a solution to \eqref{B-3} with a communication of the form \eqref{B-4-1}. For sufficiently small $\varepsilon > 0$, we have the following assertions. 
		\begin{enumerate}
			\item If $p_i$ and $p_j$ collide at $T>0$, then
			\[
				C\varepsilon \leq |q_i(T \pm \varepsilon)-q_j(T \pm \varepsilon)|,
			\]
			for some positive constant $C>0$.
			\item If $p_i$ and $p_j$ stick at $T>0$,
			\[
				D_1\varepsilon^{\frac{1}{\alpha}} \leq |q_i(T-\varepsilon)-q_j(T-\varepsilon)| \leq D_2\varepsilon^{\frac{1}{\alpha}},
			\]
			where
			\[
				D_1 = \left(\frac{2\kappa m_{G'}\alpha}{N(1-\alpha)}\right)^{\frac{1}{\alpha}},
				\quad D_2 = \left( \frac{\kappa M_{G'}2^\alpha \alpha}{1-\alpha} \right)^{\frac{1}{\alpha}}.
			\]
		\end{enumerate}
	\end{lemma}
	
	\begin{proof} Throughout the proof, we set
	\[
		\mathcal{T}=\{t_1,t_2,\cdots,t_c\}, \quad t_1 < t_2 < \cdots < t_c, \quad t_0=0,
	\]
	where $\mathcal{T}$ is defined as Proposition \ref{P3.3}. \vspace{.5cm}
		
	\indent ($\bullet$ Proof of (1)). Without loss of generality, set $\nu_i > \nu_j$ and $q_j^0 > q_i^0$. Suppose that $q_i$ collide with $q_j$ at $t_C \in \mathcal{T}$. First, we use the mean value theorem to observe

\begin{align}\begin{aligned}\label{X-0}
	\frac{d}{dt} (q_i-q_j)|_{t=t_C}
	&=G\left(\nu_i + \frac{\kappa}{N}\sum_{k=1}^N \Psi(q_{k}-q_i) \right)\bigg|_{t=t_C}
	-G\left( \nu_j+ \frac{\kappa}{N}\sum_{k=1}^N \Psi(q_{k}-q_j) \right)\bigg|_{t=t_C} \\
%	&=\frac{\kappa G'(y_{ij})}{N}\left( \nu_i - \nu_j + \sum_{k=1}^N \Psi(q_k-q_i)-\Psi(q_k-q_j)\right)\bigg|_{t=t_C} \\
	&=G'(y_{ij})(\nu_i-\nu_j) \geq m_{G'}(\nu_i-\nu_j)=:v_{ij}>0.
\end{aligned}\end{align}
Then from the continuity of the solution, for some $\delta>0$ we have
\[
	\frac{v_{ij}}{2} \leq G(p_{i}(t))-G(p_j(t)), \quad t \in [t_C-\delta, t_C+\delta].
\]
Therefore, as $q_{i}(t_C)-q_{j}(t_C)=0$, for any $0 < \varepsilon \leq \delta$, we obtain
\begin{align*}%\label{X-0'}
	|q_i(t_C \pm \varepsilon)-q_j(t_C \pm \varepsilon)|
	\geq \left| \int_0^\varepsilon  G(p_i(t_C \pm s)) - G(p_j(t_C \pm s)) ds \right|
	\geq \frac{\varepsilon v_{ij}}{2}.
\end{align*} % \vspace{.5cm}
	
	($\bullet$ Proof of (2)). Suppose that, with suitable reordering of indices, $q_1, q_2, \cdots, q_{l-1}, q_l$ \emph{starts }to stick at time $t_S \in \mathcal{T}$ $(S \in [c])$ simultaneously, and set
	\[
		q_1(t) < q_2(t) < \cdots < q_{l-1}(t) < q_l(t), \quad \text{for any} ~ t \in (t_{S-1},t_S).
	\]
	Let $i:=j+1(i,j \in [l])$. We use the mean-value theorem twice to obtain
\begin{align*}%\begin{aligned}\label{X-0-1}
	\frac{d}{dt} (q_i-q_j)
	&=G\left(\nu_i + \frac{\kappa}{N}\sum_{k=1}^N \Psi(q_{k}-q_i) \right)
	-G\left( \nu_j+ \frac{\kappa}{N}\sum_{k=1}^N \Psi(q_{k}-q_j) \right) \\
	&=\frac{\kappa G'(y_{ij})}{N}\left(\sum_{k=1}^N \big( \Psi(q_k-q_i)-\Psi(q_k-q_j) \big) \right)\\
	&=\frac{\kappa G'(y_{ij})}{N}\left(\sum_{k=1}^N \psi(z_{ijk})(q_j-q_i)\right) \\
	&= - \left(  \frac{\kappa G'(y_{ij})}{N}\sum_{k=1}^N\psi(z_{ijk}) \right) (q_i-q_j),
%\end{aligned}
\end{align*}
where $z_{ijk}$ is located between $q_k-q_i$ and $q_k-q_j$. Note that the second mean value theorem is valid since $i$ and $j$ are consecutive, so that $\Psi$ is differentiable in the interval $(q_k-q_i, q_k-q_j)$. In particular, for $k \in \{i,j\}$, $z_{ijk}$ is specifically 
\[
	\psi(z_{ijj})=\psi(z_{iji})=\frac{\Psi(q_i-q_j)}{q_i-q_j}=\frac{1}{(1-\alpha)(q_i-q_j)^\alpha}.
\]
Therefore, we have
\[
	\frac{d}{dt}(q_i-q_j) \leq -\frac{\kappa m_{G'}}{N}(\psi(z_ {ijj})+\psi(z_{iji}))(q_i-q_j) = -C_1(q_i-q_j)^{1-\alpha}, 
	\quad C_1 := \frac{2\kappa m_{G'}}{N(1-\alpha)},
\]
and $C_1>0$ is independent of initial data. Now we recall the following ODE:
\[
	\dot{x} = -C_1x^{1-\alpha}, \quad x(0)=x^0 > 0
	\quad \Longrightarrow \quad
	x(t) =
		(C_1\alpha)^{\frac{1}{\alpha}}\left(\frac{(x^0)^\alpha}{C_1\alpha}-t\right)^{\frac{1}{\alpha}}, \quad  t \in \left( 0, \frac{(x^0)^\alpha}{C_1\alpha}\right). \\
\]
Then, by the comparison principle, for any sufficiently small $\delta > 0$, we have
\begin{align}\label{X-1}
	q_i(t_S-\varepsilon)-q_j(t_S-\varepsilon)
	\leq (C_1\alpha)^{\frac{1}{\alpha}}\left(\frac{(q_i(t_S-\delta)-q_j(t_S-\delta))^\alpha}{C_1\alpha}-(\delta-\varepsilon)\right)^{\frac{1}{\alpha}}, \quad \varepsilon \in [0,\delta].
\end{align}
Now we claim that
\begin{align}\label{X-2}
	q_i(t_S-\varepsilon)-q_j(t_S-\varepsilon)
	\geq (C_1\alpha\varepsilon)^{\frac{1}{\alpha}}, \quad \varepsilon \in [0,\delta].
\end{align}
Suppose that \eqref{X-2} does not hold. Then for some $\varepsilon^* \in [0,\delta]$, we have
\[
	q_i(t_S-\varepsilon^*)-q_j(t_S-\varepsilon^*)
	< (C_1\alpha\varepsilon^*)^{\frac{1}{\alpha}}.
\]
Then \eqref{X-1} under $\delta=\varepsilon^*$ yields
\[
	q_i(t_S-\varepsilon)-q_j(t_S-\varepsilon)
	< (C_1\alpha\varepsilon)^{\frac{1}{\alpha}}, \quad \varepsilon \in [0,\varepsilon^*].
\]
Then, as the inequality is strict and $(C_1\alpha\varepsilon)^{\frac{1}{\alpha}}|_{\varepsilon=0}=0$, there exists $\varepsilon^{**} \in (0,\varepsilon^*]$ satisfying
\[
	q_i(t_S-\varepsilon^{**})-q_j(t_S-\varepsilon^{**})=0.
\]
However, since $t_S$ is a time that $q_i$ and $q_j$ \emph{starts} to stick, this is awkward, verifying \eqref{X-2}. Since choice of $i$ and $j=i-1$ was arbitrary and $C_1$ is independent of indices, we have
\begin{align*}%\label{X-2-1}
	\min_{i \neq j, i,j \in [l]} |q_i(t_S-\varepsilon)-q_j(t-\varepsilon)| %\gtrsim \varepsilon^{\frac{1}{\alpha}}
	\geq (C_1\alpha\varepsilon)^{\frac{1}{\alpha}} =: D_1\varepsilon^{\frac{1}{\alpha}}.
\end{align*}

Now take any $i',j' \in [l]$ with $i'>j'$. Since $\psi(|r|)$ decreasing in $|r|$, we have
\begin{align*}%\label{X-3}
	\Psi(q_k-q_{j'})-\Psi(q_k-q_{i'})
	= \int_{q_k-q_{i'}}^{q_k-q_{j'}} \psi(r) dr
	\leq \int_{-\frac{q_{i'}-q_{j'}}{2}}^{\frac{q_{i'}-q_{j'}}{2}} \psi(r) dr
	= \frac{2^\alpha}{1-\alpha}(q_{i'}-q_{j'})^{1-\alpha}.
\end{align*}
We apply the mean value theorem to get
 \begin{align*}%\begin{aligned}\label{X-4}
	\frac{d}{dt}({q}_{i'}-{q}_{j'})
	&= \frac{\kappa G'(y_{i'\ell})}{N}
	\left(\sum_{k=1}^N \big( \Psi(q_k-q_{i'})-\Psi(q_k-q_{j'}) \big) \right) \\
	&\geq -\frac{\kappa M_{G'}2^\alpha}{1-\alpha}
	(q_{i'}-q_{j'})^{1-\alpha}=:-C_2(q_{i'}-q_{j'})^{1-\alpha}, \quad t \in (t_S-\varepsilon,t_S),
%\end{aligned}
\end{align*}
for a positive constant $C_2>0$ independent of $i'$ and $j'$. We then apply similar technique to derive \eqref{X-2} to yield
\begin{align*}
	q_{i'}(t_S-\varepsilon)-q_{j'}(t_S-\varepsilon)
	\leq (C_2\alpha\varepsilon)^{\frac{1}{\alpha}},
	\quad \varepsilon \in [0,\delta],
\end{align*}
and therefore
\begin{align*}%\label{X-4}
	\max_{i \neq j, i,j \in [l]} |q_i(t_S-\varepsilon)-q_j(t-\varepsilon)| %\gtrsim \varepsilon^{\frac{1}{\alpha}}
	\leq (C_2\alpha\varepsilon)^{\frac{1}{\alpha}}=:D_2\varepsilon^{\frac{1}{\alpha}}.
\end{align*}

\end{proof}

\begin{lemma}\label{L3.2}
	Suppose that $0 \leq f \in L^1_{\mathrm{loc}}(\bbr)$ and $u$ is absolutely continuous on $[a,b]$. If $(f \circ u) \times u' \in L^1([a,b])$, then
	\[
		\int_{u(a)}^{u(b)}f(x)dx = \int_a^b f(u(t))u'(t)dt.
	\]   
\end{lemma}
\begin{proof}
	First suppose that $0 \leq f$ is bounded and measurable. For some constant $c$, define
	\[
		F(x) := \int_c^x f(t)dt,
	\]
	Then from boundedness of $f$, we have $F \in C^{0,1}(\bbr)$. Thus $F \circ u$ is absolutely continuous and
	\[
		(F \circ u)'(t) = f'(u(t))u'(t),
	\]
	for almost every $t \in [a,b]$. Therefore we have
	\begin{align*}
		\int_a^b f(u(t))u'(t)dt &= \int_a^b (F \circ u)'(t)dt = (F \circ u)(b) - (F \circ u)(b) \\
		&= F(u(b))-F(u(a)) = \int_{u(a)}^{u(b)}F'(x)dx = \int_{u(a)}^{u(b)}f(x)dx.
	\end{align*}
	Now suppose that $0 \leq f \in L^1_{\mathrm{loc}}(\bbr)$. Define an approximating function $f_n$ as
	\[
		f_n(x):=
		\begin{cases}
			f(x), \quad &\text{if} ~ 0 \leq f(x) \leq n, \\
			0, \quad &\text{if} ~ f(x) > n.
		\end{cases}
	\]
	Then since $f_n$ is bounded, we have 
		\begin{align*}
		\int_a^b f_n(u(t))u'(t)dt = \int_{u(a)}^{u(b)}f_n(x)dx.
	\end{align*}
	From the integrability of $f$ and $f(u(t))|u'(t)|$, we have a desired result from the dominated convergence theorem.
\end{proof}

As a direct consequence, we can establish the equivalence between \eqref{B-2} and \eqref{B-3} whenever $\psi(q_i-q_j)(G(p_i)-G(p_j))$ is locally integrable for each $i,j \in [N]$. More precisely, let $\Psi(\cdot)$ be an antiderivative of $\psi$:
\[
	\Psi(x) := \int_0^x \psi(y) dy,  \quad x \in \bbr,
\]
as long as $\psi$ is locally integrable. Let $(P,Q)$ be a solution to \eqref{B-2}, where $p_i \in W^{2,1}([0,T])$ for any $T>0$. Then it follows that
\begin{align*}
\frac{d}{dt} \Psi(q_k(t) - q_i (t)) &= \frac{d}{dt} \int_{q_k(0) - q_i (t)}^{q_k(t) - q_i (t)} \psi(y) dy \\
&= \frac{d}{dt}\int_0^t \psi(q_k(t) - q_i (t))(G(q_k(t)) - G(q_i(t)))dt && (\because \text{Lemma \ref{L3.2}}) \\
&= \psi(q_k(t) - q_i (t)) (G(p_k(t)) - G(p_i(t))),
\end{align*}
for almost every $t$. Hence, it follows from $\eqref{B-2}$ that
\begin{equation*}
 \frac{d}{dt} \left( p_i - \frac{\kappa}{N}\sum_{k=1}^N   \Psi(q_k(t) - q_i (t)) \right) = 0, \quad i \in [N],
 \end{equation*}
for almost every $t$. Now, we integrate above with respect to $t$ to get 
\begin{align*}
p_i(t) &= p_i^0 -  \frac{\kappa}{N}\sum_{k=1}^N  \Psi(q_k^0 - q_i^0)  + \frac{\kappa}{N}\sum_{k=1}^N  \Psi(q_k(t) - q_i (t)) \\
         &=: \nu_i +  \frac{\kappa}{N}\sum_{k=1}^N  \Psi(q_k(t) - q_i (t)).
\end{align*}
Conversely, if each $q_i$ is continuously differentiable and $\dot{q}_i$ is absolutely continuous in any finite time interval, we recover \eqref{B-2} from \eqref{B-3} for almost every $t$ by direct differentiation. \vspace{.5cm}

Now we are ready to prove Theorem \ref{T3.1}.

\begin{proof}[Proof of Theorem \ref{T3.1}] Let $\mathcal{T}=(t_1,t_2,\cdots,t_c)$ and $t_0=0, t_{c+1}=\infty$, where $t_i$ is increasing with respect to indices. Let $[l]$ be a set of indices sticking at $t_S \in \mathcal{T}$ as in the proof of Lemma \ref{L3.1}. For any $\varepsilon>0, k \in [c]$ and $i,j \in [N]$, we have either
		\begin{align}\label{X-5}
			q_i(t) \equiv q_j(t) ~ \text{for} ~ t \geq t_k, \quad \text{or} \quad \inf_{t \in (t_k+\varepsilon, t_{k+1}-\varepsilon)}|q_i(t)-q_j(t)|>C>0,
		\end{align}
		for some constant $C>0$ from Proposition \ref{P3.2}. Thus $\Psi(q_i(s)-q_j(s))$ is continuously differentiable for $s$ where $|s-t_k| > \varepsilon, ~ t_k \in \mathcal{T}$. Therefore by Lemma \ref{L3.2}, \eqref{B-2} and \eqref{B-3} are equivalent in time $T \in (t_k+\varepsilon, t_{k+1}-\varepsilon)$. Now consider a bounded regular communcation $\tilde{\psi}$ satisfying
		\[
			\psi(x) = \tilde{\psi}(x), \quad x \in {(C,\infty)},
		\]
		and its antiderivative $\tilde{\Psi}(x):=\int_0^x \tilde{\psi}(r)dr$. Let $T \in (t_k+\varepsilon, t_{k+1}-\varepsilon)$. For the former case of \eqref{X-5}, we have
		\[
			\Psi(q_i(T)-q_j(T)) - \Psi(q_i(t_k+\varepsilon)-q_j(t_k+\varepsilon)) = 0 = \tilde{\Psi}(q_i(T)-q_j(T)) - \tilde{\Psi}(q_i(t_k+\varepsilon)-q_j(t_k+\varepsilon)).
		\]
		For the latter case of \eqref{X-5}, since $\psi$ and $\tilde{\psi}$ are same in $(C,\infty)$, we have
		\begin{align*}
			\Psi(q_i(T)-q_j(T)) -& \Psi(q_i(t_k+\varepsilon)-q_j(t_k+\varepsilon)) = \int_{q_i(t_k+\varepsilon)-q_j(t_k+\varepsilon)}^{q_i(T)-q_j(T)} \psi(r) dr \\
			&= \int_{q_i(t_k+\varepsilon)-q_j(t_k+\varepsilon)}^{q_i(T)-q_j(T)} \tilde{\psi}(r) dr
			= \tilde{\Psi}(q_i(T)-q_j(T)) - \tilde{\Psi}(q_i(t_k+\varepsilon)-q_j(t_k+\varepsilon)).
		\end{align*}
		This yields
		\begin{align*}
	p_i(T) &= \nu_i +  \frac{\kappa}{N}\sum_{k=1}^N  \Psi(q_k(T) - q_i (T)) \\
	&= p_i(t_k+\varepsilon) -  \frac{\kappa}{N}\sum_{k=1}^N  \Psi(q_k(t_k+\varepsilon) - q_i(t_k+\varepsilon))  + \frac{\kappa}{N}\sum_{k=1}^N  \Psi(q_k(T) - q_i (T)) \\
	&= p_i(t_k+\varepsilon) -  \frac{\kappa}{N}\sum_{k=1}^N  \tilde{\Psi}(q_k(t_k+\varepsilon) - q_i(t_k+\varepsilon))  + \frac{\kappa}{N}\sum_{k=1}^N  \tilde{\Psi}(q_k(T) - q_i (T)) \\
	&= p_i(t_k+\varepsilon) + \frac{\kappa}{N}\sum_{k=1}^N \int_{t_k+\varepsilon}^T \tilde{\psi}(q_k(s)-q_i(s))(G(p_k(s))-G(p_i(s)))ds.
	\end{align*}
	Therefore, even if we change the kernel of \eqref{B-2} from $\psi$ to $\tilde{\psi}$ at time $t_S+\varepsilon$, $p_i$ still remains as a solution of a differential equation in time $(t_k+\varepsilon, t_{k+1}-\varepsilon)$. Now consider a differential equation
	\[
		\dot{\tilde{q}}_i = G(\tilde{\nu}_i + \frac{\kappa}{N}\sum_{k=1}^N\tilde{\Psi}(\tilde{q}_k(t)-\tilde{q}_i(t))),
		\quad \tilde{\nu}_i := 	\tilde{p}_i^0 - \frac{\kappa}{N}\sum_{k=1}^N(\tilde{\Psi}(\tilde{q}_k^0-\tilde{q}_i^0)),
		\quad \tilde{q}_i^0=\tilde{q}_i(0),
	\]
	where $\tilde{q}_i^0=q_i(t_k+\varepsilon), ~  \tilde{p}_i^0=p_i(t_k+\varepsilon)$.
	Since \eqref{X-5} also holds for $\tilde{q_i}$, by the same argument, we can replace $\tilde{\psi}$ to $\psi$ as well. In other word, for $t \in (t_k+\varepsilon, t_{k+1}-\varepsilon)$, a value of solution is independent of the value of $\psi$ near the origin. Therefore we may assume that $\psi$ is regular and use Lemma \ref{L2.1} for any $t \notin \mathcal{T}$. As a choice of $\varepsilon>0$ is arbitrary, we have
\begin{align*}
	&\frac{d}{dt} \| P \|_{[l]} \leq
	-\frac{\kappa m_{G'}l}{N}\psi(\|Q\|_{[l]})\|P\|_{[l]} + \frac{2\kappa M_{G'}(N-l)P^0_ML_{\psi, [l]}}{N}\|Q\|_{[l]}, \\
	&L_{\psi,[l]}(t) := \sup_{\substack{r,s \geq {q}_{[l]}(t), \\ r \neq s}} \left| \frac{\psi(r)-\psi(s)}{r-s} \right| <\infty,
	\quad {q}_{[l]}(t):= \min_{\substack{i' \in [l] \\ j' \notin [l]}}|q_{i'}(t)-q_{j'}(t)|,
\end{align*}
for $t \notin \mathcal{T}$. As $p_i$ and $p_j$ stick at $t_S$, Lemma \ref{L3.1} yields
			\[
				D_1\varepsilon^{\frac{1}{\alpha}} \leq |q_i(t_S-\varepsilon)-q_j(t_S-\varepsilon)| \leq D_2\varepsilon^{\frac{1}{\alpha}},
			\]
			where
			\[
				D_1 = \left(\frac{2\kappa m_{G'}\alpha}{N(1-\alpha)}\right)^{\frac{1}{\alpha}},
				\quad D_2 = \left( \frac{\kappa M_{G'}2^\alpha \alpha}{1-\alpha} \right)^{\frac{1}{\alpha}}.
			\]
Since $[l]$ is a set of sticking particles, particles $p_i$ and $p_j$ with indices $i \in [l]$ and $j \notin [l]$ are either collisional or separated at time $t_S$. Therefore for $0 < \delta \ll 1$, Lemma \ref{L3.1} yields
\begin{align*}
	&L_{\psi,[l]}(t_S-\delta) \leq \Big|\frac{d}{dx}\frac{1}{x^\alpha}\Big|\Bigg|_{x=C\delta} = {\alpha}C^{-\alpha-1}\delta^{-\alpha-1} \\
	&\| Q(t_S-\delta) \|_{[l]} \leq \sqrt{(l^2-l) \times (D_2\delta^{\frac{1}{\alpha}})^2} \leq lD_2\delta^{\frac{1}{\alpha}}.
\end{align*}
Therefore for $0 < \delta \ll 1$,
\begin{align*}
	\frac{d}{dt} \| P(t_S-\delta) \|_{[l]}
	%&\leq -\frac{\kappa m_{G'}}{N} \left( l \psi(\|Q\|_{[l]})
	%+ (N-l)\psi(Q_{[l]}) \right)\|P\|_{[l]}
	%+ \frac{2\kappa M_{G'}(N-l)P^0_ML_{\psi, [l]}}{N}\|Q\|_{[l]} \\
	&\leq -\frac{\kappa m_{G'}l^{1-\alpha}D_2^{-\alpha}}{N}\delta^{-1}\|P(t-\delta)\|_{[l]}
	+ \frac{2\kappa M_{G'}(N-l)P^0_M}{N} C^{-1-\alpha}\delta^{-1-\alpha}l D_2\delta^{\frac{1}{\alpha}} \\
	&=: -K_1\delta^{-1}\|P(t-\delta)\|_{[l]} + K_2 \delta^{-1-\alpha+\frac{1}{\alpha}}.
\end{align*}
	Then for fixed $\varepsilon>0$ and $t_* < t - \varepsilon$, the Gr\"onwall inequality yields
	\begin{align*}
		\|P(t-\varepsilon)\|_{[l]} \leq \exp&\left(-K_1\int_{t^*}^{t_S-\varepsilon} \frac{1}{t_S-s} ds\right) \\
		& \times \bigg[ \|P(t^*)\|_{[l]} + K_2\int_{t^*}^{t_S-\varepsilon} \exp\left( K_1 \int_{t^*}^s \frac{1}{t_S-u} du \right) (t_S-s)^{-1-\alpha+\frac{1}{\alpha}} ds \bigg] \\
		=& \frac{\varepsilon^{K_1}}{(t_s-t^*)^{K_1}}
		\times \bigg[ \|P(t^*)\| + K_2(t_s-t^*)^{K_1}\int_{t^*}^{t_S-\varepsilon}(t_S-s)^{-1-\alpha+\frac{1}{\alpha}-K_1} ds \bigg] \\
		&\lesssim \begin{cases}
			\varepsilon^{K_1}(\log(t_S-t^*)-\log\varepsilon), \quad &\text{if} ~ -\alpha+\frac{1}{\alpha} -K_1 = 0, \\
			\varepsilon^{K_1}((t_s-t^*)^{-\alpha+\frac{1}{\alpha}-K_1} - \varepsilon^{-\alpha+\frac{1}{\alpha}-K_1}), &\text{if} ~ -\alpha+\frac{1}{\alpha} -K_1 \neq 0.
		\end{cases}
	\end{align*}
		Therefore, if $q_i$ and $q_j$ starts to stick at time $t_S$, then for $\varepsilon \ll 1$,
		\begin{align*}
			\psi(q_i-q_j)(G(p_i)-G(p_j))|_{t=t_S-\varepsilon} &\leq M_{G'}\psi((D_1\varepsilon^{\frac{1}{\alpha}}))D_{P,[l]}(t-\varepsilon) \\
			&\leq M_{G'}\psi((D_1\varepsilon^{\frac{1}{\alpha}}))\|P(t-\varepsilon)\|_{[l]} \\
			&\lesssim \begin{cases}
				\varepsilon^{K_1-1}(1-\log\varepsilon), \quad &\text{if} ~ -\alpha+\frac{1}{\alpha} -K_1 = 0, \\
				|\varepsilon^{K_1-1} - \varepsilon^{-1-\alpha+\frac{1}{\alpha}}|, &\text{if} ~ -\alpha+\frac{1}{\alpha} - K_1 \neq 0.
			\end{cases}
		\end{align*}
		As $G(p_i(t))-G(p_j(t)) \equiv 0$ for $t \geq t_S$, we have
		\begin{align*}
			\psi(q_i-q_j)(G(p_i)-G(p_j))|_{[t_S-\varepsilon,t_S+\varepsilon]}
			\in L^p, \quad \text{where} \quad p \in \bigg[1, \frac{1}{\max\{ 0,1-K_1,1+\alpha-1/\alpha\}}\bigg).
		\end{align*}
		Furthermore, if $1+\alpha-1/\alpha \leq 0$ and $1 \leq K_1$, then we can choose $p = \infty$. Note that $K_1$ is increasing in $l$, the number of simultaneously sticking particles, so that
		\[
			K_1 = \frac{\kappa m_{G'}l^{1-\alpha}D_2^{-\alpha}}{N}
			= \frac{m_{G'}l^{1-\alpha}(1-\alpha)}{NM_{G'}2^\alpha\alpha}
			\geq \frac{m_{G'}2^{1-2\alpha}(1-\alpha)}{NM_{G'}\alpha}=:K. 
		\] 

		If $q_i$ and $q_j$ collide at time $t_S$, where $\nu_i > \nu_j$, by similar calculation in \eqref{X-0} we have
		\[
			m_{G'}(\nu_i-\nu_j) \leq G(p_i(t_S))-G(p_j(t_S)) \leq M_{G'}(\nu_i-\nu_j),
		\]
		and $G(p_i(t_S))-G(p_j(t_S))$ is nonzero bounded in a neighborhood of $t_S$. Together with Lemma \ref{L3.1}, this yields
		\[
			\psi(q_i-q_j)(G(p_i)-G(p_j))|_{t=t_S \pm \varepsilon} \lesssim \varepsilon^{-\alpha}.
		\]
		In this case, we have
		\begin{align*}
			\psi(q_i-q_j)(G(p_i)-G(p_j))|_{[t_S-\varepsilon,t_S+\varepsilon]}
			\in L^p, \quad \text{where} \quad p \in \bigg[1, \frac{1}{\alpha}\bigg).
		\end{align*}
		If $p_i$ and $p_j$ are neither collisional nor sticking at $t_S$, then $\psi(q_i-q_j)(G(p_i)-G(p_j))$ is bounded near $t_S$. On the other hand, $\psi(q_i(s)-q_j(s))(G(p_i(s))-G(p_j(s)))$ is bounded for $s$ such that $|s-t_c| > \varepsilon, ~ t_c \in \mathcal{T}$. Putting the results altogether, we conclude
		\[
			\sum_{i=1}^N \psi(q_i-q_j)(G(p_i)-G(p_j)) \in L^p_{\mathrm{loc}}(\bbr_+),
			\quad \text{where} \quad p \in \bigg[1, \frac{1}{\max\{ 1-K, \alpha\}}\bigg).
		\]
		This proves the first assertion of the Theorem \ref{T3.1}. \vspace{.5cm}
		
		To prove the second assertion, we consider the solution to \eqref{B-2} emanating from an initial data $(P(T),Q(T))$, where $T>t_c$. Then from Proposition \ref{P3.2}, we have either
		\begin{align*}%\label{X-6}
			q_i(t) \equiv q_j(t) ~ \text{for} ~ t > T, \quad \text{or} \quad \inf_{t \geq T}|q_i(t)-q_j(t)|>C>0,
		\end{align*}
		for some constant $C>0$. Therefore, as we did in the beginning of the proof, a value of solution for $t > T$ is independent of the value of $\psi$ near the origin. Therefore we may assume that $\psi$ is regular and apply Theorem \ref{T2.1}. Since $\int_r^\infty \psi(x) dx = \infty$ for any $r>0$, we conclude that flocking emerges unconditionally.
\end{proof}

\section{Analysis under strongly singular communications}\label{sec:4} 
\setcounter{equation}{0}

In this section, we consider strongly singular communications, which are not integrable near the origin. A typical example is
\begin{align}\label{C-1}
	\text{$\displaystyle\psi(q)=\frac{1}{|q|^\alpha}, \quad$ where $\quad \alpha \geq 1$}.
\end{align}
As in the previous section, the well-posedness of a solution is directly related to singularity arising from a collision. For a strongly singular kernel case, this issue can be treated by the so-called `collision avoidance property' of the strongly singular kernel.

\begin{proposition}\label{P4.1}
	Suppose that
	\[
		\psi \in (C^{0,1}_{\mathrm{loc}} \cap L^1_{\mathrm{loc}})(\bbr_+;\bbr_+),
		\quad (\psi(r)-\psi(s))(r-s) \leq 0, \quad \forall r,s \in \bbr_+,
	\]
	and let $Q$ be a solution of \eqref{A-3} with noncollisional initial data  $(P^0, Q^0)$. If $\int_0^\varepsilon \psi(r) dr = \infty$ for any $\varepsilon>0$, then there exists a unique global classical solution with the collision avoidance property:
	\[
		\inf_{t \in [0,T]}\min_{\substack{i,j \in [N] \\ i \neq j}}|q_i(t)-q_j(t)| > 0, \quad \forall T \in \bbr_+.
	\]
Furthermore, if the ambient space is one-dimensional($d=1$), then we have
	\[
		\inf_{t \geq 0}\min_{\substack{i,j \in [N] \\ i \neq j}}|q_i(t)-q_j(t)| > 0.
	\]
\end{proposition}

\begin{proof} Although the Proposition can proved by direct modification of \cite[Theorem 5.2]{BHK2} and \cite[Theorem 3.1]{BHK}, we provide a more simple proof here. Since $Q^0$ is non-collisional, from the standard Cauchy-Lipschitz theory, a solution is well-posed before the first collision time $\tau$ (i.e. the smallest $\tau>0$ satisfying $q_i(\tau)=q_j(\tau)$ for some $i,j,i \neq j$). To establish the global well-posedness, we first observe that the collision does not happen in any finite time. Suppose that the first collision time $\tau \in \bbr_+$ exists and let $q_i$ be a colliding particle. By the rearrangement of indices suppose define a set of indices $[l] \subset [N]$ as
\[
	[l] := \{ j \in [N] \mid q_i(\tau)=q_j(\tau) \} \neq \emptyset.
\]
Then, for any $\varepsilon>0$, we can apply the second estimate in Lemma \ref{L2.1} in time $t \in [0,\tau-\varepsilon)$. From the definition of $[l]$, there exists two positive constants $C_1,C_2>0$ satisfying
 \begin{align}\label{C-2}
		\frac{d}{dt}\| P \|_{[l]} \leq -C_1\psi( \| Q \|_{[l]})\| P \|_{[l]} + C_2\| Q \|_{[l]}, \quad  t \in [0,\tau-\varepsilon).
\end{align}
Now define the functional
\[
\tilde{\mathcal{L}}(t):= \frac{{C_1}}{ M_{G'}} \int_{ \|Q^0\|_{[l]}}^{ \|Q(t)\|_{[l]}}\psi(s)\,ds.
\]
Then $|\tilde{\mathcal{L}}(t)|+\|P(t)\|_{[l]}$ have a linear or sub-linear growth; there exists a positive constant $C$ satisfying
\begin{align*}
\frac{d}{dt}|&\tilde{\mathcal{L}}(t)| + \frac{d}{dt}\|P(t)\|_{[l]}
\leq \left| \frac{d}{dt}\tilde{\mathcal{L}}(t) \right| + \frac{d}{dt}\|P(t)\|_{[l]}\\
&= \left| \frac{{C_1}}{M_{G'}}\psi( \|Q(t)\|_{[l]})\frac{d}{dt}\|Q(t)\|_{[l]}\right| +\frac{d}{dt}\|P(t)\|_{[l]}
\leq C_2\|Q\|_{[l]} < C < \infty,
\end{align*}
where we used \eqref{A-8} and \eqref{C-2} for the second inequality.  Therefore if collision happen, there exists a constant $C$ satisfying
\[
	\infty = \lim_{t \nearrow \tau}|\tilde{\mathcal{L}}(t)| \lesssim C(1+\tau) < \infty,
\]
which yields a contradiction. Therefore collision cannot happen in any finite time, and this proves the existence and uniqueness of a global classical solution. \vspace{.5cm}

Now suppose $d=1$. From \eqref{A-3}, we can deduce an integral equation
\[
	G^{-1}(\dot{q}_i(t))=G^{-1}(\dot{q}_i(0))+\frac{\kappa}{N}\int_{0}^t \sum_{k=1}^N \alpha_{ki}(s)(G^{-1}(\dot{q}_k(s))-G^{-1}(\dot{q}_i(s))) ds, \quad t \in \bbr_+,
\]
where the modified kernel $\alpha_{ki}$ is defined as
\begin{align*}
	\alpha_{ki}(t) =
	\begin{cases}
		0 \quad &\text{if} ~ \dot{q}_k(t)=\dot{q}_i(t), \\
		\displaystyle \phi(q_k(t)-q_i(s))\times \frac{\dot{q}_k(t)-\dot{q}_i(t)}{G^{-1}(\dot{q}_k(t))-G^{-1}(\dot{q}_i(t))} \quad &\text{if} ~ \dot{q}_k(t) \neq \dot{q}_i(t).
	\end{cases}
\end{align*}
so that $\alpha_{ki}$ is nonnegative, measurable and symmetric with respect to indices. Since a finite-in-time collision never happens, a solution is well defined globally, and \cite[Theorem 1]{HT} guarantees the existence the following uniform-in-$t$ bound $U$ for each $i,j \in [N] (i \neq j)$:
\begin{align*}
	\left| \int_{q_i^0-q_j^0}^{q_i(t)-q_j(t)} \psi(r) dr \right|
	&= \left| \int_0^{t} \psi(q_i(s)-q_j(s))(\dot{q}_i(s)-\dot{q}_j(s)) ds \right| \\
	&= \left| \int_{0}^t \alpha_{ij}(s)(G^{-1}(\dot{q}_i(s))-G^{-1}(\dot{q}_j(s))) ds \right| \\
	&\leq \int_{0}^\infty \alpha_{ij}(s)|(G^{-1}(\dot{q}_i(s))-G^{-1}(\dot{q}_j(s))| ds =: U < \infty.
\end{align*}
As $\psi$ is not integrable near the origin, we conclude
	\[
		\sup_{t \geq 0} \max_{\substack{i,j \in [N] \\ i \neq j }} \left| \int_{q_i^0-q_j^0}^{q_i(t)-q_j(t)} \psi(r) dr \right| \leq U
		\quad \Longrightarrow \quad \inf_{t \geq 0} \min_{\substack{i,j \in [N] \\ i \neq j}}|q_i(t)-q_j(t)| > 0.
	\]
\end{proof}

\begin{remark} ~
	\begin{enumerate}
	\item From Proposition \ref{P4.1}, the origin of the kernel is not referred to in any finite time. Therefore, under the same assumption in Proposition \ref{P4.1}, although $\psi$ is not Lipschitz, the results of Theorem \ref{T2.1} still hold.
	\item Although the proof of Proposition \ref{P4.1} is rather simple, explicit lower bounds between agents cannot be deduced. For the explicit expression for a lower bound, refer to the proof of \cite[Theorem 5.2]{BHK2}. 
	\item If the kernel is weakly singular at the origin (i.e., $\int_0^\varepsilon \psi(r) dr < \infty$ for some $\varepsilon>0$), then a collision might happen, as described in the previous section. % For the details, refer to, for example, \cite{BHK}.
	\end{enumerate}
\end{remark}

For the Euclidean space of arbitrary dimension, the authors of \cite{YYC} derived an existence of strict positive lower bound for relative distances under $\alpha > 2$ and $G=\mathrm{Id}$:
\[
	\min_{i,j \in [N]}|q_i^0-q_j^0| > 0
	\quad \Rightarrow \quad 
	\inf_{t \geq 0}\min_{i,j \in [N]}|q_i(t)-q_j(t)| \geq L_\infty > 0,
\]
 by employing a suitable potential energy with a dissipative structure. Unfortunately, the dissipation of potential energy heavily depends on the Galilean invariance, which \eqref{A-3} lacks due to the presence of a velocity control function. Instead, we provide an alternative characterization for the existence of $L_\infty$.

\begin{theorem}\label{T4.1}
	Let $(P,Q)$ be a solution of \eqref{A-3} with a kernel of the form \eqref{C-1}. Suppose that $Q$ is non-collisional and further assume that
	\begin{enumerate}
		\item $\alpha \neq 1$, and
		\item $(P,Q)$ exhibits flocking.
	\end{enumerate}
	Then there exists a strictly positive lower bound of distance between agents:
	\[
		\inf_{t \geq 0}\min_{\substack{i,j \in [N] \\ i \neq j}}|q_i(t)-q_j(t)| > 0.
	\]
\end{theorem}

\begin{proof}
	From Proposition \ref{P4.1}, the result of Theorem \ref{T2.1} holds for kernel of the form \eqref{C-1} as well. Therefore there exists a constant $C>0$ satisfying
	\[
		\max_{i,j}|p_i(t)-p_j(t)| \lesssim e^{-tC}, \quad t \in \bbr_+,
	\]
	and the limit $\lim_{t \to \infty}|q_i(t)-q_j(t)|$ always exists for any indices. In particular if $\lim_{t \to \infty}|q_i(t)-q_j(t)| = 0$, then we have
	\begin{align}\begin{aligned}\label{C-3}
		|q_i(t)-q_j(t)| &= \left| \int_\infty^t (\dot{q}_i(t)-\dot{q}_j(t)) dt \right|
		= \left| \int^\infty_t (\dot{q}_j(t)-\dot{q}_i(t)) dt \right| \\
		&\leq M_{G'}\int^\infty_t | p_i(s) - p_j(s) |ds
		 \lesssim \int^\infty_t e^{-sC} ds \lesssim e^{-tC}.
	\end{aligned}\end{align}
	Now for some index $i$, suppose that there exists a set $[l] \subset [N]$ defined as
	\[
		[l] := \{ j \in [N] \mid \lim_{t \to \infty} (q_i(t)-q_j(t))=0 \} \neq \emptyset.
	\]
	From \eqref{C-3}, there exist positive constants $B,C>0$ satisfying $\|Q(t)\|_{[l]} \leq Be^{-tC}$. Therefore for sufficiently large $t \gg 1$, we obtain
	\begin{align*}
	\left|\int_{\|Q^0\|_{[l]}}^{\|Q(t)\|_{[l]}} \psi(s) ds\right|
	= \int^{\|Q^0\|_{[l]}}_{\|Q(t)\|_{[l]}} \psi(s) ds \geq \int^{\|Q^0\|_{[l]}}_{Be^{-tC}} \psi(s) ds
	= \frac{B^{1-\alpha}e^{tC(\alpha-1)}-\|Q^0\|_{[l]}^{1-\alpha}}{\alpha-1}.
	\end{align*}
	Since the limit of $|p_i(t)-p_j(t)|$ always exists and finite-in-time collision never happens, there exists a positive constant $U$ satisfying
	\[
		\sup_{t \geq 0}\sup_{\substack{i \in [l] \\ j \notin [l]}}\psi(q_i-q_j) < U < \infty.
	\]
	Therefore $|\tilde{\mathcal{L}}|$ defined in the proof of Proposition \ref{P4.1} have a linear or sub-linear growth, which leads to the contradiction:
	\begin{align*}
	e^{tC(\alpha-1)}-1 \lesssim \left|\int_{\|Q^0\|_{[l]}}^{\|Q(t)\|_{[l]}} \psi(s) ds\right| \lesssim 1+t.
	\end{align*}
	Therefore we conclude $[l]=\emptyset$, as desired.
\end{proof}

\begin{remark}\label{R4.2}
	\begin{enumerate}
	\item Let $(P,Q)$ be a solution to \eqref{A-3} with non-collisional initial data $(P^0,Q^0)$. Corollary \ref{C2.1} states a high value of $\kappa$ leads to not only flocking but also the strict spacing between the agents; 
	\begin{align}\label{C-4}
		\min_{\substack{i,j \in [N] \\ i \neq j}}\inf_{t \geq 0}|q_i(t)-q_j(t)| > 0
	\end{align}
	Moreover, when communication is of form $\psi(x)=|x|^{-\alpha}$, the theorem in \cite{YYC} states that \eqref{C-4} can be achieved for arbitrary $\kappa>0$ under $\alpha > 2$ and $G=\mathrm{Id}$. However, to the author's knowledge, a further result to have \eqref{C-4} under $\alpha \in [1,2]$ is missing. Meanwhile, Theorem \ref{T2.1} and \ref{T4.1} states that $\kappa$ can be arbitrarily small when $\alpha$ is close to 1. Therefore Theorem 4.1 may complement the previous result.
%	\item The second assumption in Theorem \ref{T4.1} can be replaced by the following assumption.
%	\begin{center}
%	\hspace{-2cm} (2') The magnitude of relative states does not oscillate:
%	\[
%		\lim_{t \to \infty}|q_i(t)-q_j(t)| \in \{ \bbr_{\geq 0}, +\infty\}, \quad \forall ~ i,j \in [N],
%	\]
%	\end{center}
%	and the flocking condition (2) is a sufficient condition for (2').
	\item If a strictly positive lower bound of the relative state is guaranteed, as we did in the proof of Theorem \ref{T3.1}, we may regularize the kernel. As an application, for example, we may apply results of stability estimates in \cite{HKZ} for singular kernels as well, even though proof of the theorem requires the Lipschitz continuity of the kernel. On the other hand, we may relax a priori condition for stability estimate in \cite{A}, since some of the conditions are devoted to ensuring strict lower bound between relative states.
	\item In many cases concerning a many-body system equipped with a singular kernel, it is often desirable to guarantee a strictly positive lower bound for a relative distance between agents. This, for example, guarantees the well-definedness of a $\omega$-limit set, which enables us to apply the dynamical system theory like LaSalle's invariance principle.
	\end{enumerate}
\end{remark}

\section{conclusion}\label{sec:5}

In this paper, we have presented the CS consensus model with a velocity control function. We provided several conditions to have the emergence of mono or bi-cluster flocking. When communication is singular, the model exhibits some interesting properties not found in regular communication, like sticking or collision avoidance, depending on the integrability of the kernel near the origin. If communication is weakly singular, the particles can stick or collide in finite time, which leads to the loss of regularity. We have studied the existence, uniqueness, and regularity of such solutions on the real line. On the other hand, when a communication is strongly singular, the particles never collide in a finite time whenever initial data is non-collisional. However, the existence of a strictly lower bound for relative distance for the general initial data was left as a remaining issue. We proved that a prototypical communication with strong singularity $(\psi(|q|)=|q|^{-\alpha}, ~ \alpha \geq 1)$ leads to a strictly positive lower bound between agents, provided that singularity is not critical($\alpha \neq 1$) and flocking is guaranteed. Several interesting problems remain as future perspectives. For example, relaxing a priori conditions for strict spacing between the agents is one of the remaining problems. On the other hand, our study of the sticking solution was limited to the real line, and it would be interesting to extend the result to a more general space. 
	
%%%%%%%%%%%%%%%%%%%%%%%%%%%%%%%%%%%%%%%%%%%%%%%%%%%%%%%%


\begin{thebibliography}{99}

\bibitem{A} H. Ahn: \textit{Uniform stability of the Cucker–Smale and thermodynamic Cucker–Smale ensembles with singular kernels.}  Netw. Heterog. Media. {\bf 17}(5) (2022), 753.

\bibitem{ABHY} H. Ahn, J. Byeon, S.-Y. Ha and J. Yoon: \textit{On the relativistic flocks over the unit sphere and the hyperboloid in a bonding force field}. J. Math. Phys. {\bf 64} (2023), Paper No. 012705.

%\bibitem{A-B}  J. A. Acebron, L. L. Bonilla, C. J. P. P\'{e}rez Vicente, F. Ritort and R. Spigler.: \textit{The Kuramoto model: A simple paradigm for synchronization phenomena.} Rev. Mod. Phys. {\bf77} (2005), 137-185.

%\bibitem{AHK21+} H. Ahn, S.-Y. Ha, and J. Kim: \textit{Uniform stability of the Euclidean relativistic Cucker-Smale model and its application to a mean-field limit}. Commun. Pure Appl. Anal. {\bf 20} (2021), 4209--4237.

\bibitem{AHK} H. Ahn, S.-Y. Ha, and J. Kim: \textit{Nonrelativistic limits of the relativistic Cucker-Smale model and its kinetic counterparts}. J. Math. Phys. {\bf 63} (2022), Paper No. 082701.

\bibitem{AH} S. Ahn and S.-Y. Ha: \textit{Stochastic flocking dynamics of the Cucker–Smale model with multiplicative white noises.}  J. Math. Phys.  {\bf 51}(2010), Paper No. 103301

%\bibitem{AHKS} H. Ahn, S.-Y. Ha, M. Kang, and W. Shim: \textit{Emergent behaviors of relativistic flocks on Riemannian manifolds}. Phys. D: Nonlinear Phenom. {\bf427} (2021), 133011.

\bibitem{AB19} G. Albi, N. Bellomo, L. Fermo, S.-Y. Ha, J. Kim, L. Pareschi, D. Poyato and J. Soler: \textit{Vehicular traffic, crowds, and swarms. From kinetic theory and multiscale methods to applications and research perspectives}. Math. Models Methods Appl. Sci. {\bf 29} (2019), 1901--2005.

%\bibitem{B-B} J. Buck and E. Buck.: \textit{Biology of synchronous flashing of fireflies.} Nature {\bf211} (1966), 562-564.

%\bibitem{BCD}
%J. C. Bronski, T. E. Carty and L. Deville: \textit{Synchronisation conditions in the Kuramoto model and their relationship to seminorms.}
%Nonlinearity {\bf 34} (2021), 5399–-5433.

\bibitem{BHK2} J. Byeon, S.-Y. Ha, and J. Kim: \textit{Emergence of state-locking for the first-order nonlinear consensus model on the real line}.

\bibitem{BHK} J. Byeon, S.-Y. Ha, and J. Kim: \textit{Asymptotic flocking dynamics of a relativistic Cucker-Smale flock under singular communications}. J. Math. Phys. {\bf 63} (2022), Paper No. 012702.

\bibitem{CCMP17} J. A. Carrillo, Y.-P. Choi, P. B. Mucha, and J. Peszek, \textit{Sharp conditions to avoid collisions in singular Cucker-Smale interactions}. Nonlinear Anal. Real World Appl. {\bf 37} (2017), 317--328.

\bibitem{CCTT16} J. A. Carrillo, Y.-P. Choi, E. Tadmor, and C. Tan: \textit{Critical thresholds in 1D Euler equations with non-local forces}. Math. Models Methods Appl. Sci. {\bf 26} (2016), 185--206.

%\bibitem{CCH14} J. A. Carrillo, Y.-P. Choi, and M. Hauray, \textit{Local well-posedness of the generalized Cucker-Smale model with singular kernels}. Mathematical modelling of complex systems, 17--35, ESAIM Proc. Surveys {\bf 47}, EDP Sci., Les Ulis, 2014.

%\bibitem{CCMP17} J. A. Carrillo, Y.-P. Choi, P. B. Mucha, and J. Peszek, \textit{Sharp conditions to avoid collisions in singular Cucker-Smale interactions}. Nonlinear Anal. Real World Appl. {\bf 37} (2017), 317--328.

%\bibitem{CD14} F. Cucker and J.-G. Dong, \textit{A conditional, collision-avoiding, model for swarming}. Discrete Contin. Dyn. Syst. {\bf 34} (2014), 1009--1020.

%\bibitem{CD11} F. Cucker and J.-G. Dong, \textit{A general collision-avoiding flocking framework}. IEEE Trans. Automat. Control {\bf 56} (2011), 1124--1129.

%\bibitem{CD10} F. Cucker and J.-G. Dong, \textit{Avoiding collisions in flock}. IEEE Trans. Automat. Control {\bf 55} (2010), 1238--1243.

\bibitem{CFRT10} J. A. Carrillo, M. Fornasier, J. Rosado, and G. Toscani: \textit{Asymptotic flocking dynamics for the kinetic Cucker-Smale model}. SIAM J. Math. Anal. {\bf 42} (2010), 218--236.

\bibitem{C-F-T-V} J. A. Carrillo, M. Fornasier, G. Toscani and F. Vecil, \textit{Particle, kinetic, and hydrodynamic models of swarming. Mathematical modeling of collective behavior in socio-economic and life sciences.} (2010), pp. 297--336, Model. Simul. Sci. Eng. Technol., Birkhauser Boston, Inc., Boston, MA.

%\bibitem{CHK18} Y.-P. Choi, S.-Y. Ha, and J. Kim, \textit{Propagation of regularity and finite-time collisions for the thermomechanical Cucker-Smale model with a singular communication}. Netw. Heterog. Media {\bf 13} (2018), 379--407.

%\bibitem{CKPP19} Y.-P. Choi, D. Kalise, J. Peszek, and A. A. Peters, \textit{A collisionless singular Cucker-Smale model with decentralized formation control}. SIAM J. Appl. Dyn. Syst. {\bf 18} (2019), 1954--1981.

\bibitem{CHHJK} J. Cho, S.-Y. Ha, F. Huang, C. Jin and D. Ko: \textit{Emergence of bi-cluster flocking for the Cucker–Smale model.} Mathematical Models and Methods in Applied Sciences, {\bf 26}(06) (2016), 1191-1218.

\bibitem{CH} S. Choi and S.-Y. Ha: \textit{Emergence of flocking for a multi-agent system moving with constant speed.} Commun. Math. Sci. \textbf{14} (2016), 953--972.

\bibitem{CHL17} Y.-P. Choi, S.-Y. Ha, and Z. Li: \textit{Emergent dynamics of the Cucker-Smale flocking model and its variants.} Active particles. Vol. 1. Advances in theory, models and applications, 299--331, Model. Simul. Sci. Eng. Technol., Birkhauser/Springer, Cham, 2017.

%\bibitem{C-Z} Y. P. Choi and X. Zhang: \textit{One dimensional singular Cucker-Smale model: uniform-in-time mean-field limit and contractivity.} J. Differential Equations {\bf 287} (2021), 428-459.

\bibitem{CS07-2} F. Cucker and S. Smale: \textit{Emergent behavior in flocks.} IEEE Trans. Autom. Control {\bf 52} (2007), 852--862.

\bibitem{CS07-1} F. Cucker and S. Smale: \textit{On the mathematics of emergence}. Japan J. Math. {\bf 2} (2007), 197--227.

%\bibitem{D-M1} P. Degond and S. Motsch.: \textit{Large-scale dynamics of the persistent turning walker model of fish behavior.} J. Stat. Phys. {\bf 131} (2008), 989-1022. 

\bibitem{HKR20} S.-Y. Ha, J. Kim, and T. Ruggeri: \textit{From the relativistic mixture of gases to the relativistic Cucker-Smale flocking.} Arch. Rational Mech. Anal. \textbf{235} (2020), 1661--1706.

%\bibitem{HKR} S.-Y. Ha, J. Kim, and T. Ruggeri, \textit{Kinetic and hydrodynamic models for the relativistic Cucker-Smale ensemble and emergent behaviors}. To appear in Commun. Math. Sci.


\bibitem{HKPZ19} S.-Y. Ha, J. Kim, J. Park, and X. Zhang: \textit{Complete cluster predictability of the Cucker-Smale flocking model on the real line}. Arch. Ration. Mech. Anal. {\bf 231} (2019), 319--365.

\bibitem{HKZ} S.-Y. Ha, J. Kim and X. Zhang: \textit{Uniform stability of the Cucker-Smale model and its application to the mean-field limit.} Kinet. Relat. Models, {\bf 11}(5) (2018), 1157.

%\bibitem{HKoZ} S.-Y. Ha, D. Ko and Y. Zhang: \textit{Critical coupling strength of the Cucker–Smale model for flocking.} Math. Models Methods Appl. Sci. {\bf 27}(06), 1051-1087.

%\bibitem{H-J-K-K} S.-Y. Ha, S. Jin, D. Kim and D. Ko: \textit{Convergence toward equilibrium of the first-order consensus model with random batch interactions.} J. of Differential Equations {\bf 302} (2021), 585--616.

%\bibitem{H-J}  Ha, S.-Y. and Jung, J.: \textit{A hybrid fractional Kuramoto model and its emergent behavior.} Preprint. 

\bibitem{HPZ18}  S.-Y. Ha, J. Park, and X. Zhang: \textit{A first-order reduction of the Cucker-Smale model on the real line and its clustering dynamics}. Commun. Math. Sci. {\bf 16} (2018),  1907--1931.

\bibitem{HL09} S.-Y. Ha and J.-G. Liu: \textit{A simple proof of Cucker-Smale flocking dynamics and mean field limit}. Commun. Math. Sci. {\bf 7} (2009), 297--325.

%\bibitem{HR17} S.-Y. Ha and T. Ruggeri, \textit{Emergent dynamics of a thermomechanically consistent particle model}. Arch. Rational Mech. Anal. {\bf 233} (2017), 1397--1425.

\bibitem{HT08} S.-Y. Ha and E. Tadmor: \textit{From particle to kinetic and hydrodynamic description of flocking}. Kinet. Relat. Models {\bf 1} (2008), 415--435.

\bibitem{HT} J. M. Hendrickx and J.N. Tsitsiklis: \textit{Convergence of type-symmetric and cut-balanced consensus seeking systems.} IEEE Trans. Automat. Contr. {\bf58}(1) (2012), 214-218.

%\bibitem{K-H-S} D. Ko, S.-Y. Ha and W. Shim: \textit{Emergent dynamics of random Cucker-Smale model on a line.} To appear in Studies in Applied Mathematics.

%\bibitem{K84} Y. Kuramoto: \textit{Chemical Oscillations, waves and turbulence.} Springer-Verlag, Berlin (1984).

\bibitem{K21} J. Kim: \textit{First-order reduction and emergent behavior of the one-dimensional kinetic Cucker-Smale equation}. J. Differential Equations {\bf 302} (2021), 496--532.

%\bibitem{K22+} J. Kim: \textit{A Cucker-Smale flocking model with the Hessian communication weight and its first-order reduction}. J. Nonlinear Sci. {\bf 32} (2022), Article No. 20.

%\bibitem{LFM} Z. Lin, B. Francis and M. Maggiore: \textit{State agreement for continuous-time coupled nonlinear systems.} SIAM J. Control Optim. {\bf 46}(1) (2007), 288-307.

%\bibitem{MMPZ19} P. Minakowski, P. B. Mucha, J. Peszek, and E. Zatorska, \textit{Singular Cucker-Smale dynamics}. Active Particles, Vol. 2., 201--243, Model. Simul. Sci. Eng. Technol., Birkh\"auser/Springer, Cham, 2019.

\bibitem{M-A-H-K} C. H. Min, H. Ahn, S.-Y. Ha and M. Kang: \textit{Sufficient conditions for asymptotic phase-locking to the generalized Kuramoto model}. Kinet. Relat. Models, {\bf 16}(1) (2023), 97-132.

\bibitem{M-T1} S. Motsch and E. Tadmor: \textit{Heterophilious dynamics enhances consensus.} SIAM Rev. {\bf 56} (2014), 577--621.

\bibitem{MP18} P. B. Mucha and J. Peszek, \textit{The Cucker-Smale equation: singular communication weight, measure-valued solutions and weak-atomic uniqueness}. Arch. Rational Mech. Anal. {\bf 227} (2018), 273--308.

%\bibitem{P15} J. Peszek, \textit{Discrete Cucker-Smale flocking model with a weakly singular weight}. SIAM J. Math. Anal. {\bf 47} (2015), 3671--3686.

%\bibitem{P14} J. Peszek, \textit{Existence of piecewise weak solutions of a discrete Cucker-Smale's flocking model with a singular communication weight}. J. Differential Equations {\bf 257} (2014), 2900--2925.

%\bibitem{PS17} D. Poyato and J. Soler, \textit{Euler-type equations and commutators in singular and hyperbolic limits of kinetic Cucker-Smale models.} Math. Models Methods Appl. Sci. {\bf 27} (2017), 1089--1152.

\bibitem{P15} J. Peszek, \textit{Discrete Cucker--Smale flocking model with a weakly singular weight}. SIAM Journal on Mathematical Analysis, (2015), 47(5), 3671-3686.

\bibitem{P14} J. Peszek, \textit{Existence of piecewise weak solutions of a discrete Cucker–Smale's flocking model with a singular communication weight}, J. Differential Equations, (2014), 257(8), 2900-2925.


\bibitem{P-R-K} A. Pikovsky, M. Rosenblum and J. Kurths:  \textit{Synchronization: A universal concept in nonlinear sciences.} Cambridge University Press, Cambridge, 2001.

%\bibitem{Sa} S. Saks: \textit{Theory of the Integral}, Dover Publications, 1937.

%\bibitem{S} J. Szarski: \textit{Differential inequalities.} Instytut Matematyczny Polskiej Akademii Nauk, 1965.

%\bibitem{T} T. Tao: \textit{Nonlinear dispersive equations: local and global analysis}. No. 106. American Mathematical Soc., 2006.

%\bibitem{T-B} C. M. Topaz and A. L. Bertozzi.: \textit{Swarming patterns in a two-dimensional kinematic model for biological groups.} SIAM J. Appl. Math. {\bf 65} (2004), 152-174.

%\bibitem{TH} T. Chen and H Chen: \textit{Universal approximation to nonlinear operators by neural networks with arbitrary activation functions and its application to dynamical systems.} IEEE Trans. Neural Netw. Learn. Syst., {\bf 6}(4) (1995), 911-917.

\bibitem{T-T} J. Toner and Y. Tu: \textit{Flocks, herds, and schools: A quantitative theory of flocking.} Phys. Rev. E {\bf 58} (1998), 4828-4858.

%\bibitem{V-C-B-C-S} T. Vicsek, A. Czir\'ok, E. Ben-Jacob, I. Cohen and O. Schochet: \textit{Novel type of phase transition in a system of self-driven particles.} Phys. Rev. Lett.  {\bf 75} (1995), 1226-1229.

\bibitem{VZ} T. Vicsek and A. Zefeiris.: \textit{Collective motion.} Phys. Rep. {\bf 517} (2012), 71-140.

%\bibitem{W67} A. T. Winfree, \textit{Biological rhythms and the behavior of populations of coupled oscillators.} J. Theor. Biol. {\bf 16} (1967), 15–42.

\bibitem{YYC} X. Yin, D. Yue and Z. Chen: \textit{Asymptotic behavior and collision avoidance in the Cucker–Smale model.} IEEE Trans. Automat. Control, {\bf 65}(7) (2019), 3112-3119.

%\bibitem{Z-Z} X. Zhang and T. Zhu: \textit{Complete classification of the asymptotical behavior for singular C-S model on the real line.} J. Differential Equations {\bf 269} (2020),  201-256.

\end{thebibliography}
\end{document}